\documentclass{article}
\usepackage{graphicx}
\usepackage[utf8]{inputenc}
\usepackage[T1]{fontenc}
\usepackage{amsmath,amssymb,lmodern}
\usepackage{amsthm}
\usepackage{caption}
\usepackage{color,soul}
\providecommand{\keywords}[1]{\textbf{Keywords} #1}

\newtheorem{remark}{Remark}
\newtheorem{result}{Result}
\newtheorem{theorem}{Theorem}

\begin{document}

\title{Apriori and aposteriori error estimation of Subgrid multiscale stabilized finite element method for fully coupled Navier-Stokes\\
 Transport model}

\author{B.V. Rathish Kumar, Manisha Chowdhury\thanks{ Email addresses:  drbvrk11@gmail.com (B.V.R. Kumar); chowdhurymanisha8@gmail.com(M.Chowdhury)  } }
     
\date{Indian Institute of Technology Kanpur \\ Kanpur, Uttar Pradesh, India}

\maketitle

\begin{abstract}
In this paper a fully coupled system of transient $Navier$-$Stokes$ ($NS$) fluid flow model and variable coefficient unsteady Advection-Diffusion-Reaction ($VADR$) transport model has been studied through subgrid multiscale stabilized finite element method. In particular algebraic approach of approximating the subscales has been considered to arrive at stabilized variational formulation of the coupled system. This system is strongly coupled since viscosity of the fluid depends upon the concentration, whose transportation is modelled by $VADR$ equation. Fully implicit schemes have been considered for time discretisation. Further more elaborated derivations of both $apriori$ and $aposteriori$ error estimates for stabilized finite element scheme have been carried out. Credibility of the stabilized method is also established well through various numerical experiments, presented before concluding.
\end{abstract}


\section{Introduction}
For more than a decade transport equation coupled with fluid flow model attracted the attention of researchers due to its significant role in modelling various real life problems ranging from environmental issues to physiological importance. For instance contemporary world is tackling with the challenges of ground water pollution due to diffusion of pollutant transported through rivers, use of drug-eluting stents to remove stenosis in human arteries after implanting stents into them etc. Authors of $\cite{RefJ}-\cite{RefN}$ have studied various coupled systems involving different fluid flow models and transport equation. Whereas $Vassilev$ and $Yotov$ in $\cite{RefJ}$ have presented a mixed finite element analysis of coupled Stokes-Darcy-Transport model, $Hui$ and $Jhang$ in $\cite{RefM}$ have discussed about a stabilized mixed finite element method for coupled transient Stokes-Darcy flows with transport. $Cesmelio\breve{g}lu$ $et$ $al.$  have studied continuous and discontinuous finite element methods for coupled Navier-Stokes/Darcy and transport problems in $\cite{RefK}$ and $Cesmelio\breve{g}lu$ together with  $Rivi\grave{e}re$  in $\cite{RefL} $ have presented a mathematical analysis of existence and uniqueness of coupled NS-Darcy- unsteady Transport equation. Recently $Chowdhury$ and $Kumar$ $\cite{RefN}$ have considered to study subgrid scale stabilized finite element analysis of coupled Stokes-Brinkman-Transport problem. Importantly authors in $\cite{RefL}$-$\cite{RefN}$ have considered the viscosity of fluid flow problem dependent on concentration of the solute transported into the fluid. In few recent works authors in $\cite{RefW}$-$\cite{RefX}$ have focused on studying advection-diffusion transport equation coupled with incompressible Navier-Stokes equation. Whereas $Du$ and $Liu$ in $\cite{RefW}$ have worked with lattice Boltzmann model, $Yua$ $et$ $al.$ ($\cite{RefX}$) have studied finite difference method for the coupled model. Both of these studies have considered constant viscosity coefficient which indicates an one-way or weak coupling between fluid flow model and transport equation. In this paper we have presented a stabilized finite element analysis of transient Navier-Stokes ($NS$) fully-coupled with unsteady Advection-diffusion-reaction equation with variable coefficients ($VADR$). $Subgrid$ $multiscale$ ($SGS$), a most general finite element stabilization technique, has been employed to study this coupled system. We have considered concentration dependent viscosity in the fluid flow model as well as spatially variable diffusion coefficients in transport equation. These considerations make this coupling  not only two-sided or strong but also more efficient to model the contemporary real life challenges accurately. The previous studies on coupled $NS$-Transport model have neither considered variable viscosity and diffusion coefficients nor discussed about any error estimation for the method applied to study the model. In this study we have elaborately carried out both $apriori$ and $aposteriori$ error estimations for a general finite element stabilization scheme to study strongly coupled transient $NS$-$VADR$ model. \vspace{1mm}\\
It is well known fact that lack of stability in standard Galerkin finite element method has driven researchers to introduce stabilized methods such as Streamline upwind/Petrov-Galerkin ($SUPG$) formulation, Galerkin/least-squares ($G$ $LS$) method, characteristic-based split ($CBS$) method, Subgrid Scale ($SGS$) method, bubble stabilization etc. Over last four decades huge developments have been taken place in the study of various stabilization techniques. Starting with the works of $Brooks$ and $Hughes$ $\cite{RefA}$ on $SUPG$; $Hughes$, $Franca$ and $Hulbert$ $\cite{RefB}$ on $GLS$; $Hughes$ introducing $SGS$ in $\cite{RefC}$; $Russo$ $\cite{RefH}$ explaining bubble stabilization method for the linearized incompressible $NS$ equations, the stabilization schemes have been growing through the studies of $Hannani$ $et$ $al.$ $\cite{RefD}$ on comparison between $SUPG$ and $GLS$ formulation for steady state incompressible $NS$ equations; $Codina$ and $Zienkiewicz$ $\cite{RefI}$ on comparison of $CBS$ and $GLS$; $Codina$ $et$ $al.$ on numerical comparison of $CBS$ and $SGS$ method of the incompressible $NS$ equations in $\cite{RefE}$; $Russo$ $\cite{RefG}$ on comparison of $SUPG$ and residual-free bubbles($RFB$); $Kirk$ and $Carey$ $\cite{RefF}$ on development and validation of $SUPG$ for compressible $NS$ equations etc. $Codina$ in $\cite{RefO} $ has experimentally established that for solving $ADR$ equation $SGS$ method performs better than other stabilized methods, such as $SUPG$, $GLS$, $Taylor$-$Galerkin$ etc. and in fact it is the most general method amongst them. \vspace{1mm}\\
 Generally two approaches of $SGS$ stabilized formulation, namely algebraic approach, abbreviated as $ASGS$ and orthogonal projection approach, known as $OSGS$ method,  have been studied. Though few studies $\cite{RefA1}$- $\cite{RefA3}$ are there applying only general form of $SGS$ method instead of working with one specific approach, but authors in $\cite{RefB1}$-$\cite{RefB4}$ have employed $ASGS$ method, whereas authors of $\cite{RefC1}$- $\cite{RefC5}$ have worked with $OSGS$ method. Again $Badia$ and $Codina$ in $\cite{RefV}$ have studied both the approaches for unified Stokes-Darcy fluid flow problem and experimentally established equally well performances of both the stabilized formulations. In this paper we have considered algebraic approach of approximating subscales which implies the stabilization parameters are of algebraic forms. This stabilization method begins with division of weak solution space into the spaces of the known finite element space and an unknown subgrid or unresolvable scale space and finally the stabilized formulation has been reached through expressing the element of subgrid scales in terms of the element of resolvable finite element space. For time discretization fully implicit schemes have been chosen. A detailed derivation of $apriori$ error estimation has been carried out for this stabilized variational form of the coupled system. Further more residual based $aposteriori$ error estimate too is derived elaborately. First order convergence in space has been established with respect to complete norms on all the variables. This paper also establishes the accuracy of the stabilized method through various numerical results, which include all possible combinations of cases containing small and large Reynolds numbers as well as cases involving concentration dependent viscosity. In every numerical example $ASGS$ performs consistently well in compared to standard Galerkin finite element method.\vspace{1mm}\\
This paper is organised as follows: Section 2 introduces the coupled system along with important assumptions. In the next section we have introduced weak formulation, stabilized formulation and finally fully-discrete formulation after applying time-discretization rule. This section  also contains stability analysis of the fully-discrete stabilized form. Section 4 has elaborately described the derivations of $apriori$ and $aposteriori$ error estimations for this stabilized formulation. Before concluding the article section 5 presents numerical results to verify accuracy of the method.

\section{Model problem}
In this section we introduce the  flow problem described through transient $Navier$-$Stokes$ equations coupled with unsteady transport model over an open bounded domain $\Omega \subset$ $\mathbb{R}^d$, d=2,3 with piece-wise smooth boundary $\partial\Omega$. Let us first present the $Navier$-$Stokes$ fluid flow model in the following: 
Find velocity $\textbf{u}$: $\Omega$ $\times$ (0,T) $\rightarrow \mathbb{R}^d$ and pressure $p$: $\Omega \times$ (0,T) $\rightarrow \mathbb{R}$ of the fluid such that,
\begin{equation}
\begin{split}
\rho \frac{\partial \textbf{u}}{\partial t}+\rho ( \textbf{u} \cdot \bigtriangledown ) \textbf{u} - \mu(c) \Delta \textbf{u} + \bigtriangledown p & = \textbf{f} \hspace{2mm} in \hspace{2mm} \Omega \times (0,T) \\
\bigtriangledown \cdot \textbf{u} &= \textbf{0} \hspace{2mm} in \hspace{2mm} \Omega \times (0,T) \\
\textbf{u} &= \textbf{0} \hspace{2mm} on \hspace{2mm} \partial\Omega \times (0,T) \\
\textbf{u} &= \textbf{u}_0 \hspace{2mm} at \hspace{2mm} t=0 \\
\end{split}
\end{equation} 
where $\rho$ is the density of the fluid, $\mu(c)$ is the dynamic viscosity of the fluid depending on concentration $c$ of the dispersing mass of the solute,  $\textbf{f}$ is body force and $\textbf{u}_0$ is the initial velocity. \vspace{2 mm}\\
This above flow problem is fully-coupled with the following transient advection-diffusion-reaction equation with variable coefficients($VADR$), representing the transportation of solute in $\Omega$.\\
Find the concentration $c$: $\Omega \times$ (0,T) $\rightarrow \mathbb{R}$ of the solute such that,
\begin{equation}
\begin{split}
 \frac{\partial c}{\partial t}- \bigtriangledown \cdot \tilde{\bigtriangledown} c + \textbf{u} \cdot \bigtriangledown c + \alpha c & = g \hspace{2mm} in \hspace{2mm} \Omega \times (0,T) \\
c &= 0 \hspace{2mm} on \hspace{2mm}\partial \Omega \times (0,T) \\
c & = c_0 \hspace{2mm} at \hspace{2mm} t=0\\
\end{split}
\end{equation}
where the notation, $\tilde{\bigtriangledown}: = \sum_{i=1}^d D_i \frac{\partial}{\partial x_i} e_i $ for $d=2,3$ and $\{e_i \}_{i=1}^d$ is standard basis of $\mathbb{R}^d$. $D_i$ are variable diffusion coefficients, $\alpha$ is the reaction coefficient and $g$ denotes the source of solute mass and $c_0$ is the initial concentration of the solute.  \vspace{1.0 mm}\\ 
Let us consider a notation \textbf{U}= (\textbf{u},p,c) and the system of equations can be written in the following operator form,
\begin{equation}
M\partial_t \textbf{U} + \mathcal{L}(\textbf{u}; \textbf{U}) = \textbf{F}
\end{equation}
where M, a matrix = diag($\rho$,$\rho$,0,1), $\partial_t \textbf{U} = (\frac{\partial \textbf{u}}{\partial t}, \frac{\partial p}{\partial t}, \frac{\partial c}{\partial t})^T$ \\
\[
\mathcal{L} (\textbf{u}; \textbf{U})=
  \begin{bmatrix}
  \rho ( \textbf{u} \cdot \bigtriangledown ) \textbf{u}  - \mu(c) \Delta \textbf{u} + \bigtriangledown p\\
    \bigtriangledown \cdot \textbf{u} \\
    - \bigtriangledown \cdot \tilde{\bigtriangledown} c + \textbf{u} \cdot \bigtriangledown c + \alpha c 
  \end{bmatrix}
\]
 and \[
\textbf{F}=
  \begin{bmatrix}
    \textbf{f} \\
    0 \\
    g
  \end{bmatrix}
\]
Let us introduce the adjoint $\mathcal{L}^*$ of $\mathcal{L}$ as follows,
\[
\mathcal{L}^* (\textbf{u}; \textbf{U})=
  \begin{bmatrix}
  -\rho ( \textbf{u} \cdot \bigtriangledown ) \textbf{u} - \mu(c) \Delta \textbf{u}  - \bigtriangledown p\\
    -\bigtriangledown \cdot \textbf{u} \\
    - \bigtriangledown \cdot \tilde{\bigtriangledown} c - \textbf{u} \cdot \bigtriangledown c + \alpha c 
  \end{bmatrix}
\]
Now we assume suitable conditions on the coefficients mentioned above, which will be useful to conclude the results further. \vspace{1mm} \\
\textbf{(i)} The fluid viscosity $\mu(c)= \mu \in C^0(\mathbb{R}^+; \mathbb{R}^+)$, the space of positive real valued functions defined on positive real numbers and we will have two positive real numbers $\mu_l$ and $\mu_u$ such that 
\begin{equation}
0 < \mu_l \leq \mu(x) \leq \mu_u \hspace{2mm} for \hspace{2mm} any \hspace{2mm} x\in \mathbb{R}^+
\end{equation}
\textbf{(ii)} $D_i= D_i(\textbf{x},t) \in C^0(\mathbb{R}^d \times (0,T);\mathbb{R})$ (for $i=1,...,d$) where $ C^0(\mathbb{R}^d \times (0,T);\mathbb{R})$ is the space of real valued continuous function defined on $\mathbb{R}^d$ for fixed $t \in (0,T)$. Both are bounded quantity that is we can find lower and upper bounds for both of them.\vspace{1mm} \\
\textbf{(iii)} $\rho$ and $\alpha$ are positive constants. \vspace{1mm}\\
\textbf{(iv)} The spaces of continuous solution $(\textbf{u},p,c)$ are assumed as: \\ $\textbf{u} \in L^{\infty}(0,T;(H^2(\Omega))^d)\bigcap C^{0}(0,T; H_0^1(\Omega))$ and \\
$p \in L^{\infty}(0,T;H^1(\Omega))\bigcap C^{0}(0,T;L^2_0(\Omega)) $,
 $c \in L^{\infty}(0,T;H^2(\Omega))\bigcap C^0(0,T;H^1_0(\Omega))$ \vspace{0.1mm}\\
\textbf{(v)} Additional assumptions on continuous velocity and concentration solutions are: $\textbf{u}_{tt}, \textbf{u}_{ttt}, c_{tt}, c_{ttt}$ all are taken to be bounded functions on $\Omega$ for each $t \in (0,T)$. \vspace{1mm}\\
\textbf{Weak formulation}: Assuming the body force $\textbf{f} \in L^2(0,T; (L^2(\Omega))^d)$ and the source term $g \in L^2(0,T; L^2(\Omega))$ let us consider the spaces suitable to define the weak form as $V=H^1_0(\Omega)$ and $Q=L^2(\Omega)$ and J= (0,T).\vspace{1mm}\\
Now denoting the space $V^d \times Q \times V$ (for $d=2,3$) by $\textbf{V}_F$ the weak formulation of (3) is to find \textbf{U}= (\textbf{u},p,c): J $ \rightarrow \textbf{V}_F$ such that $\forall$ \textbf{V}=(\textbf{v},q,d) $\in  \textbf{V}_F$
\begin{equation}
\begin{split}
(M\partial_t \textbf{U},\textbf{V}) + B(\textbf{u}; \textbf{U}, \textbf{V}) = L(\textbf{V})  
\end{split}
\end{equation} 
where 
$B(\textbf{u}; \textbf{U}, \textbf{V}) = c(\textbf{u},\textbf{u},\textbf{v})+ a_{NS}(\textbf{u},\textbf{v})- b(\textbf{v},p)+ b(\textbf{u},q)+ a_T(c,d)$ \vspace{1mm}\\
$(M \partial_t \textbf{U}, \textbf{V})= \rho \int_{\Omega} \frac{\partial \textbf{u}}{\partial t} \cdot \textbf{v}+\int_{\Omega}\frac{\partial c}{\partial t} d$ and $L(\textbf{V})= l_{NS}(\textbf{v})+ l_T(d) $ \vspace{1mm}\\
where the notations are defined in the following:\\
 $c(\textbf{u},\textbf{v},\textbf{w})=\rho \int_{\Omega} ((\textbf{u} \cdot \bigtriangledown)\textbf{v})\cdot \textbf{w}+ \frac{\rho}{2} \int_{\Omega} (\bigtriangledown \cdot \textbf{u})\textbf{v} \cdot \textbf{w} $\vspace{1 mm}\\
$a_{NS}(\textbf{u},\textbf{v})= \int_{\Omega} \mu(c) \bigtriangledown \textbf{u}:\bigtriangledown \textbf{v}$ \vspace{1 mm}\\
 $b(\textbf{v},q)= \int_{\Omega} (\bigtriangledown \cdot \textbf{v}) q$ \vspace{1 mm} \\
 $a_T(c,d) = \int_{\Omega} \tilde{\bigtriangledown}c \cdot \bigtriangledown d + \int_{\Omega} d \textbf{u} \cdot \bigtriangledown c + \alpha\int_{\Omega}cd $ \vspace{1 mm} \\
 $l_{NS} (\textbf{v})= \int_{\Omega} \textbf{f} \cdot \textbf{v}$ and  $l_T(d)= \int_{\Omega} gd$ \vspace{1 mm} \\
The modified trilinear form $c(\cdot,\cdot,\cdot)$ considered here is equivalent to the original trilinear form obtained from the non-linear convective term in (1). By the virtue of this modified form the trilinear term introduces the following important property. \vspace{1mm}\\
\textbf{(a)} For any $\textbf{u} \in V^d$, $c(\textbf{u},\textbf{v},\textbf{v})= 0$  $\forall \textbf{v} \in V^d$ \vspace{1mm}\\
Besides the trilinear form $c(\cdot,\cdot,\cdot)$ has the another property \cite{RefR} too. \vspace{1mm}\\
\textbf{(b)} For any $\textbf{u}, \textbf{v}, \textbf{w}$ $\in V^d$
 \begin{equation}
  c(\textbf{u},\textbf{v},\textbf{w})\leq \begin{cases}
    C \|\textbf{u}\|_1 \|\textbf{v}\|_1 \|\textbf{w}\|_1 & \\
    C \|\textbf{u}\|_0 \|\textbf{v}\|_2 \|\textbf{w}\|_1 & \\
    C \|\textbf{u}\|_2 \|\textbf{v}\|_1 \|\textbf{w}\|_0 & \\
    C \|\textbf{u}\|_0 \|\textbf{v}\|_1 \|\textbf{w}\|_{L^{\infty}(\Omega)} 
  \end{cases}
\end{equation}
where $C$ is a constant and $\| \cdot \|_j$ for $ j$=0,1,2 denote standard $L^2, H^1, H^2$ full norms respectively.
In the other forms the bilinear forms $a_{NS}(\cdot, \cdot)$ is $coercive$ $\cite{RefR}$ and $a_T(\cdot,\cdot)$ is also $continuous$ and $coercive$ $\cite{RefS}$. Again $b(\cdot,\cdot)$ satisfies $inf$-$sup$ condition $\cite{RefR}$ too.\vspace{1mm}\\

\section{Discrete formulation}
\subsection{Semi-discrete formulation}
In this section we introduce the finite element space discretization for the variational formulation (5) followed by a stabilized finite element formulation for the same.\vspace{1 mm} \\
Let the domain $\Omega$ be discretized into finite numbers of subdomains $\Omega_k$ for k=1,2,...,$n_{el}$, where $n_{el}$ is the total number element subdomains. Let $h_k$ be the diameter of each subdomain $\Omega_k$ and h= $\underset{k=1,2,...n_{el}}{max} h_k$ \vspace{1 mm}\\
Let $\tilde{\Omega}= \bigcup_{k=1}^{n_{el}} \Omega_k$ be the union of interior elements.\vspace{1 mm}\\
Let $V^h= \{ v \in V: v(\Omega_k)= \mathcal{P}^k(\Omega_k)\} $ and 
$Q^h= \{ q \in Q : q(\Omega_k)= \mathcal{P}^k(\Omega_k)\}$ \vspace{1 mm}\\
where $V^h$ and $Q^h$ be finite dimensional subspaces of $V$ and $Q$ respectively and $\mathcal{P}^k(\Omega_k)$  denotes complete polynomial of order $k$ over each $\Omega_k$ for k=1,2,...,$n_{el}$. For regular partitions the functions belonging to finite dimensional spaces satisfy the following inverse inequalities:\\
\begin{center}
$\|\Delta v_h\| \leq C_I h^{-1} \|\bigtriangledown v_h\|_{0,k}$ \hspace{2mm} and \hspace{2mm} $\|\bigtriangledown v_h\|_{0,k} \leq C_I h^{-1} \|v_h\|_{0,k}$
\end{center}
Considering similar notation $\textbf{V}_F^h$, denoting $\textbf{V}_F^h=(V^h)^d \times Q^h \times V^h $ the \textbf{finite element formulation}  of the variational form (5) in the finite dimensional space $\textbf{V}_F^h$ is to
find $\textbf{U}_h $= $(\textbf{u}_h,p_h,c_h)$: J $ \rightarrow \textbf{V}_F^h$ such that $\forall$ $\textbf{V}_h=(\textbf{v}_h,q_h,d_h)$ $\in \textbf{V}_F^h$
\begin{equation}
(M\partial_t \textbf{U}_h,\textbf{V}_h) + B(\textbf{u}_h; \textbf{U}_h, \textbf{V}_h) = L(\textbf{V}_h)   
\end{equation}
where $(M\partial_t \textbf{U}_h,\textbf{V}_h)$= $ \rho (\frac{\partial \textbf{u}_h}{\partial t}, \textbf{v}_{h})+ (\frac{\partial c_h}{\partial t}, d_h)$ \\
 $B(\textbf{u}_h; \textbf{U}_h, \textbf{V}_h)$ = $ c(\textbf{u}_h,\textbf{u}_h,\textbf{v}_h)+ a_{NS}(\textbf{u}_h,\textbf{v}_h)- b(\textbf{v}_h,p_h)+ b(\textbf{u}_h,q_h)+ a_T(c_h,d_h)$ \\
 and $L(\textbf{V}_h)= l_{NS}(\textbf{v}_h)+ l_T(d_h) $ \vspace{2mm}\\
 In addition let us consider the initial conditions $(\textbf{u}_h, \textbf{v}_h)\mid_{t=0}=(\textbf{u}_0, \textbf{v}_h)$ $\forall \textbf{v}_h \in (V^h)^d $ and $(c_h,d_h)\mid_{t=0}= (c_0,d_h)$ $\forall d_h \in V^h$. \vspace{1mm}\\
Now we are going to introduce $subgrid$ $multiscale$ stabilized finite element method with algebraic approximation of the subscales of (5).
It involves decomposition of the weak solution space $\textbf{V}_F$ into the spaces of resolvable scales and unresolvable or subgrid scales. The finite element space $\textbf{V}_F^h$ is chosen to be the space of resolvable scales and in literature one of the ways of choosing the space of subgrid scales is the space that completes $\textbf{V}_F^h$ in $\textbf{V}_F$. Then the final form of subgrid formulation will be arrived while the elements of subgrid scales will be expressed in the terms of elements of resolvable scales. \vspace{1 mm} \\
The \textbf{stabilized algebraic subgrid multiscale ($ASGS$) formulation} for this coupled equation to 
find $\textbf{U}_h $= $(\textbf{u}_h,p_h,c_h)$: J $ \rightarrow \textbf{V}_F^h$ such that $\forall$ $\textbf{V}_h=(\textbf{v}_h,q_h,d_h)$ $\in \textbf{V}_F^h$ 
\begin{equation}
(M\partial_t \textbf{U}_h,\textbf{V}_h) + B_{ASGS}(\textbf{u}_h ; \textbf{U}_h, \textbf{V}_h)  = L_{ASGS}(\textbf{V}_h)  
\end{equation}
where $B_{ASGS}(\textbf{u}_h ;\textbf{U}_h, \textbf{V}_h)= B(\textbf{u}_h; \textbf{U}_h, \textbf{V}_h)+ \sum_{k=1}^{n_{el}} (\tau_k'(M\partial_t \textbf{U}_h + \mathcal{L}(\textbf{u}_h ;\textbf{U}_h)-\textbf{d}), -\mathcal{L}^*(\textbf{u}_h;\textbf{V}_h))_{\Omega_k}- \sum_{k=1}^{n_{el}}((I-\tau_k^{-1}\tau_k')(M\partial_t \textbf{U}_h + \mathcal{L}(\textbf{u}_h;\textbf{U}_h)), \textbf{V}_h)_{\Omega_k} \\
-\sum_{k=1}^{n_{el}} (\tau_k^{-1}\tau_k' \textbf{d}, \textbf{V}_h)_{\Omega_k}$ \vspace{2 mm}\\
$L_{ASGS}(\textbf{V}_h)= L(\textbf{V}_h)+ \sum_{k=1}^{n_{el}}(\tau_k' \textbf{F}, -\mathcal{L}^*(\textbf{u}_h;\textbf{V}_h))_{\Omega_k}- \sum_{k=1}^{n_{el}}((I-\tau_k^{-1}\tau_k')\textbf{F}, \textbf{V}_h)_{\Omega_k}$  \vspace{1 mm} \\
where the stabilization parameter $\tau_k$ is in matrix form as 
\[
\tau_k= diag(\tau_{1k},\tau_{1k},\tau_{2k},\tau_{3k}) =
  \begin{bmatrix}
    \tau_{1k} I_{d \times d} & 0 & 0 \\
    0 & \tau_{2k} & 0 \\
    0 & 0 & \tau_{3k}
  \end{bmatrix}
\]
and 
\[
\tau_k'= (\frac{1}{dt}M+ \tau_k^{-1})^{-1} =
  \begin{bmatrix}
    \frac{\tau_{1k} dt}{dt+ \rho \tau_{1k}}I_{d \times d} & 0 & 0 \\
    0 & \tau_{2k} & 0 \\
    0 & 0 & \frac{\tau_{3k} dt}{dt+ \tau_{3k}}
  \end{bmatrix}\\
  = diag (\tau_{1k}',\tau_{1k}',\tau_{2k}',\tau_{3k}') 
\]
$I_{d \times d}$ is an identity matrix for $d=2,3$.\vspace{1 mm}\\
$\textbf{d}$= $\sum_{i=1}^{n+1}(\frac{1}{dt}M\tau_k')^i(\textbf{F} -M\partial_t \textbf{U}_h - \mathcal{L}(\textbf{u}_h ;\textbf{U}_h))$ =$[\textbf{d}_1,d_2,d_3]^T$ \vspace{1 mm}\\
It can be easily observed that $d_2$ is always 0 due to the matrix M. \vspace{2 mm}\\
We have the forms of the stabilization parameters $\tau_{1k}, \tau_{2k}$ for $Navier$-$Stokes$ equation in $\cite{RefU}$ and $\tau_{3k}$ for $VADR$ equation $\cite{RefT}$ and for each k=1,2,...,$n_{el}$ all the coefficients $\tau_{ik}$ coincide with $\tau_{i}$ for i=1,2,3 and choosing the parameters $c_1,c_2,c_3$ suitably that $\tau_{i}$'s are as follows:

\begin{equation}
\begin{split}
\tau_{1k} &= \tau_{1}= (c_1 \frac{\mu_{\textbf{u}}}{h^2}+  c_2 \frac{\rho \| \textbf{u}_h \|}{h})^{-1} \\
\tau_{2k} &=\tau_{2}=\frac{h^2}{c_1 \tau_{1}} \\
\tau_{3k} & = \tau_{3}= c_3(\frac{9D}{4h^2} + \frac{3 \| \textbf{u}_h \|}{2h} + \alpha )^{-1}
\end{split}
\end{equation}
where $ \textbf{u}_h $ is the computed velocity.
\begin{remark}
Considering continuity of the solutions at the inter-element boundaries, we have not encountered with any jump term in the above stabilized formulation.
\end{remark}

\subsection{Fully-discrete formulation}
Before introducing time discretization, some notations  have been introduced: for $dt$= $\frac{T}{N}$, where $N$ is a positive integer, $t_n= n dt$ and for given $0 \leq \theta \leq 1$,
\begin{equation}
\begin{split}
f^n & = f(\cdot , t_n) \hspace{4 mm} for \hspace{2 mm} 0 \leq n \leq N\\
f^{n,\theta} &= \frac{1}{2} (1 + \theta) f^{(n+1)} + \frac{1}{2} (1- \theta) f^n \hspace{4mm} for \hspace{2mm} 0\leq n \leq N-1
\end{split}
\end{equation}
Later we will see for $\theta=0$ the discretization follows Crank-Nicolson formula and for $\theta=1$ it is backward Euler discretization rule.\vspace{1mm}\\
For sufficiently smooth function $f(t)$, using the Taylor series expansion about t= $t^{n,\theta}$, we will have \vspace{1mm}\\
\begin{equation}
\begin{split}
f^{n+1} & = f(t^{n,\theta})+ \frac{(1-\theta)  dt}{2} \frac{\partial f}{\partial t}(t^{n,\theta}) + \frac{(1-\theta)^2 dt^2}{8} \frac{\partial^2 f}{\partial t^2} (t^{n,\theta}) + \mathcal{O}(dt^3)\\
f^{n} & = f(t^{n,\theta})- \frac{(1+\theta) dt}{2} \frac{\partial f}{\partial t}(t^{n,\theta}) + \frac{(1+\theta)^2 dt^2}{8} \frac{\partial^2 f}{\partial t^2}(t^{n,\theta}) + \mathcal{O}(dt^3)
\end{split}
\end{equation}
We have considered here $t^{n,\theta}- t^n= \frac{(1+\theta) \Delta t}{2}$\\
Multiplying the above first and second sub-equations in (14) by $\frac{1+\theta}{2}$ and $\frac{1-\theta}{2}$ respectively and then adding them we will have the following\\
\begin{equation}
f^{n,\theta} = f(t^{n,\theta}) + \frac{1}{8} (1+\theta)(1-\theta) dt^2 \frac{\partial^2 f}{\partial t^2}(t^{n,\theta}) + \mathcal{O}(dt^3)
\end{equation} 
Let $\textbf{u}^{n,\theta},p^{n,\theta},c^{n,\theta}$ be approximations of $\textbf{u}(\textbf{x},t^{n,\theta}), p(\textbf{x},t^{n,\theta}),c(\textbf{x},t^{n,\theta})$ respectively. Now by Taylor series expansion \cite{RefQ},we have 
\begin{equation}
\begin{split}
\frac{\textbf{u}^{n+1}-\textbf{u}^n}{dt} & = \frac{\partial \textbf{u}}{\partial t}(\textbf{x},t^{n,\theta}) + \textbf{TE}_1\mid_{t=t^{n,\theta}} \hspace{4mm} \forall \textbf{x} \in \Omega  \\
\frac{c^{n+1}-c^n}{dt} & = \frac{\partial c}{\partial t}(\textbf{x},t^{n,\theta}) + TE_2\mid_{t=t^{n,\theta}} \hspace{5mm} \forall \textbf{x} \in \Omega
\end{split}
\end{equation}
where the truncation error $TE\mid_{t=t^{n,\theta}}$ $\simeq$ $TE^{n,\theta}$ depends upon time-derivatives of the respective variables and $dt$.
\begin{equation}
\begin{split}
\|\textbf{TE}_1^{n,\theta}\| & \leq
      \begin{cases}
      C' dt \|\textbf{u}_{tt}^{n,\theta}\|_{L^{\infty}(t^n,t^{n+1},L^2)} & if \hspace{1mm} \theta=1 \\
    C'' dt^2 \|\textbf{u}_{ttt}^{n,\theta}\|_{L^{\infty}(t^n,t^{n+1},L^2)} & if \hspace{1mm} \theta=0
      \end{cases}
\end{split}
\end{equation}
The above relation holds for $TE_2$ in similar manner.
Now applying assumption $\textbf{(v)}$ on $\textbf{u}_{tt}$ and $\textbf{u}_{ttt}$ we will have another property as follows:
\begin{equation}
\begin{split}
\|\textbf{TE}_1^{n,\theta}\| & \leq\begin{cases}
      C' dt  & if \hspace{1mm} \theta=1 \\
    C'' dt^2 & if \hspace{1mm} \theta=0
      \end{cases}
\end{split}
\end{equation}
Similarly
\begin{equation}
\begin{split}
\| TE_2^{n,\theta}\| & \leq\begin{cases}
      C' dt  & if \hspace{1mm} \theta=1 \\
    C'' dt^2 & if \hspace{1mm} \theta=0
      \end{cases}
\end{split}
\end{equation}
After introducing all the required definitions finally the fully-discrete formulation of $subgrid$ form is as follows: For given $\textbf{U}_h^n = (\textbf{u}_h^n,p_h^n,c_h^n)\in \textbf{V}_F^h$ find $\textbf{U}_h^{n+1}= (\textbf{u}_h^{n+1},p_h^{n+1},c_h^{n+1}) \in \textbf{V}_F^h $  such that , $\forall \hspace{1mm} \textbf{V}_h=(\textbf{v}_h,q_h,d_h) \in \textbf{V}_F^h $
\begin{equation}
(M\frac{(\textbf{U}_h^{n+1}-\textbf{U}_h^n)}{dt}, \textbf{V}_h)+ B_{ASGS}(\textbf{u}_h^{n} ;\textbf{U}_h^{n,\theta}, \textbf{V}_h) = L_{ASGS}(\textbf{V}_h) + (\textbf{TE}^{n,\theta},\textbf{V}_h)  
\end{equation}
Again for the exact solution we will have the discrete formulation as follows: \\
For given $\textbf{U}^n = (\textbf{u}^n,p^n,c^n)\in \textbf{V}_F$ find $\textbf{U}^{n+1}= (\textbf{u}^{n+1},p^{n+1},c^{n+1}) \in \textbf{V}_F $  such that , $\forall \hspace{1mm} \textbf{V}_h=(\textbf{v}_h,q_h,d_h) \in \textbf{V}_F^h$
\begin{equation}
(M\frac{(\textbf{U}^{n+1}-\textbf{U}^n)}{dt}, \textbf{V}_h)+ B(\textbf{u}^{n} ;\textbf{U}^{n,\theta}, \textbf{V}_h) = L(\textbf{V}_h) + (\textbf{TE}^{n,\theta},\textbf{V}_h) 
\end{equation}

\section{Error estimates}
We start this section with introducing the projection operator corresponding to each unknown variable followed by notation of error and it's component wise splitting. Later we go to derive $ apriori$ and $aposteriori$ error estimates.

\subsection{Projection operators : Error splitting}
Let us introduce the projection operator for each of these error components.\vspace{1 mm}\\
(i)For any $\textbf{u} \in (H^2(\Omega))^d  $  we assume that there exists an interpolation $I^h_{\textbf{u}}:  (H^2(\Omega))^d \longrightarrow  (V^h)^d $ satisfying $b(\textbf{u}-I^h_{\textbf{u}}\textbf{u}, q_h)=0$ \hspace{2mm} $\forall q_h \in Q^h$ \vspace{2mm}\\
(ii) Let $I^h_p: H^1(\Omega) \longrightarrow Q^h$ be the $L^2$ orthogonal projection given by \\ $\int_{\Omega}(p-I^h_pp)q_h=0$  \hspace{1mm} $\forall q_h \in Q^h$ and for any $p \in H^1(\Omega)$ \vspace{2mm}\\
(iii)Similarly let $I^h_{c}: H^2(\Omega) \longrightarrow V^h$ be the $L^2$ orthogonal projection given by \\ $\int_{\Omega}(c-I^h_c c)d_h=0$ \hspace{1mm} $ \forall d_h \in V^h$ and for any $c \in H^2(\Omega)$ \vspace{2mm}\\
Let $\textbf{e}=(e_{\textbf{u}},e_p,e_c)$ denote the error where the components are $e_{\textbf{u}}=\textbf{u}-\textbf{u}_h, e_p= (p-p_h)$ and $e_c=(c-c_h)$. Now each component of the error can be split into two parts interpolation part, $E^I$ and auxiliary part, $E^A$ as follows: \vspace{1mm}\\
$e_{\textbf{u}}=(\textbf{u}-\textbf{u}_{h})=(\textbf{u}-I^h_{\textbf{u}}\textbf{u})+(I^h_{\textbf{u}}\textbf{u}-\textbf{u}_{h})= E^{I}_{\textbf{u}}+ E^{A}_{\textbf{u}}$ \vspace{1mm}\\
Similarly
$e_{p}=E^{I}_{p}+ E^{A}_{p}$, and 
$e_{c}=E^{I}_{c}+ E^{A}_{c}$ \vspace{2mm}\\
At this point let us mention the standard \textbf{interpolation estimation} result \cite{RefQ} in the following: for any exact solution with regularity upto (m+1)
\begin{equation}
\|v-I^h_v v\|_l = \|E^I_v\|_l \leq C(p,\Omega) h^{m+1-l} \|v\|_{m+1} 
\end{equation}
where l ($\leq m+1$) is a positive integer and C is a constant depending on m and the domain. For l=0 and 1 it  implies standard $L^2(\Omega)$ and $H^1(\Omega)$ norms respectively. For simplicity we will use $\| \cdot \|$ instead of $\| \cdot \|_0$ to denote $L^2(\Omega)$ norm.
Now we put some results using the properties of projection operators and these results will be used in error estimations.
\begin{result}
\begin{equation}
(\frac{\partial}{\partial t} E^{I,n}, v_h)=0 \hspace{2mm} v_h \in V^h
\end{equation}
\end{result}

\begin{result}
For any given auxiliary error $E^{A,n}$ and unknown $E^{A,n+1}$
\begin{equation}
(\frac{\partial}{\partial t} E^{A,n}, E^{A,n,\theta}) \geq  \frac{1}{2 dt} (\|E^{A,n+1}\|^2- \|E^{A,n}\|^2)
\end{equation}
\end{result}

\begin{remark}
The proof of the results have been discussed in  \cite{RefN} elaborately.
\end{remark}
 
\subsection{Apriori error estimate}
In this section we will find $apriori$ error bound, which depends on the exact solution. Here we first estimate $auxiliary$ error bound and later using that we will find $apriori$ error estimate. Before deriving error estimations let us define norms required for error estimations. Let us consider the space $\tilde{\textbf{V}} $ := $L^2(0,T; V_s)\bigcap L^{\infty}(0,T; Q_s)$ and it's associated norm is denoted by $\tilde{\textbf{V}}$-norm. For the functions $g_1, g_2, g_3$ belonging to the spaces $L^2(0,T; L^2(\Omega))$, $L^2(0,T; H_0^1(\Omega))$,  $\tilde{\textbf{V}}$ respectively norms over these spaces, abbreviated as $L^2(L^2)$, $L^2(H^1)$, $\tilde{\textbf{V}}$ are defined in the following
\begin{equation}
\begin{split}
\|g_1\|_{L^2(L^2)}^2 &= \sum_{n=0}^{N-1} \int_{t^n}^{t^{n+1}} \int_{\Omega} \mid g_1^{n,\theta} \mid^2 dt \\
\|g_2\|_{L^2(H^1)}^2 &= \sum_{n=0}^{N-1} \int_{t^n}^{t^{n+1}} (\int_{\Omega} \mid g_2^{n,\theta} \mid^2  +  \int_{\Omega} \mid \frac{\partial g_2}{\partial x}^{n,\theta} \mid^2  +  \int_{\Omega} \mid \frac{\partial g_2}{\partial y}^{n,\theta} \mid^2 )dt \\
\|g_3\|_{\tilde{\textbf{V}}}^2 & = \underset{0\leq n \leq N}{max} \|g_3^n\|^2 + \|g_3\|_{L^2(H^1)}^2\\
\end{split}
\end{equation}
\begin{theorem} \textbf{(Auxiliary error estimate)} 
For computed velocity $\textbf{u}_h$, pressure $p_h$ and concentration $c_h$ belonging to $(V^h)^d \times Q^h \times V^h$ satisfying (31)-(32), assume $dt$ is sufficiently small and positive, and sufficient regularity of exact solution in equations (1)-(2). Then there exists a constant C, depending upon $\textbf{u},p,c$ such that
\begin{equation}
\|E^A_{\textbf{u}}\|^2_{\tilde{\textbf{V}}} + \|E^A_p\|_{L^2(L^2)}^2  + \|E^A_{c}\|^2_{\tilde{\textbf{V}}} \leq C (h^2+ dt^{2r})
\end{equation}
where
\begin{equation}
    r=
    \begin{cases}
      1, & \text{if}\ \theta=1 \\
      2, & \text{if}\ \theta=0
    \end{cases}
  \end{equation}
\end{theorem}
\begin{proof} In first part we will find bound for auxiliary error part of velocity $\textbf{u}$ and concentration c with respect to ${\tilde{\textbf{V}}}$-norm and in the second part we will estimate auxiliary error for pressure term with respect to $Q$ norm and finally combining them we will arrive at the desired result. \vspace{2mm} \\
\textbf{First part} Subtracting (17) from (18) and then simplifying the terms, we have $\forall \hspace{1mm} \textbf{V}_h \in (V^h)^d \times Q^h \times V^h$

\begin{multline}
(M\frac{(\textbf{U}^{n+1}-\textbf{U}^{n+1}_{h})- (\textbf{U}^{n}-\textbf{U}^{n}_{h})}{dt},\textbf{V}_{h}) + B(\textbf{u}^{n};\textbf{U}^{n,\theta}, \textbf{V}_h)-B(\textbf{u}_h^{n};\textbf{U}^{n,\theta}_h, \textbf{V}_h)\\
+ \sum_{k=1}^{n_{el}}(\tau_k'(M\partial_t (\textbf{U}^n-\textbf{U}^n_h)+ \mathcal{L}(\textbf{u}^{n} ;\textbf{U}^{n,\theta})-\mathcal{L}(\textbf{u}^{n}_h ;\textbf{U}^{n,\theta}_h)),-\mathcal{L}^* (\textbf{u}_h;\textbf{V}_h))_{\Omega_k} \\
+\sum_{k=1}^{n_{el}}((I-\tau_k^{-1}\tau_k')(M \partial_t(\textbf{U}^{n}-\textbf{U}^{n}_h) + \mathcal{L}(\textbf{u}^{n} ;\textbf{U}^{n,\theta})-\mathcal{L}(\textbf{u}^{n}_h ;\textbf{U}^{n,\theta}_h)), -\textbf{V}_h)_{\Omega_k}\\
+ \sum_{k=1}^{n_{el}} (\tau_k^{-1}\tau_k' \textbf{d}, \textbf{V}_h)_{\Omega_k}+\sum_{k=1}^{n_{el}}(\tau_k' \textbf{d},-\mathcal{L}^*(\textbf{u}_h; \textbf{V}_h))_{\Omega_k}=(\textbf{TE}^{n,\theta}, \textbf{V}_h)\\
\end{multline}
where $\textbf{d}$= $(\sum_{i=1}^{n+1}(\frac{1}{dt}M\tau_k')^i)(M\partial_t (\textbf{U}^n-\textbf{U}^n_h) + \mathcal{L}(\textbf{u}^{n,\theta} ;\textbf{U}^{n,\theta})-\mathcal{L}(\textbf{u}_h^{n,\theta} ;\textbf{U}_h^{n,\theta}))$ \vspace{2mm}\\
Now after applying error splitting for each of the terms and later using the result obtained in (31) and properties of projection operators we have rearranged the above equation (36) as follows: $\forall  \textbf{V}_h \in (V^h)^d \times Q^h \times V^h$
\begin{multline}
\rho (\frac{E^{A,n+1}_{\textbf{u}}-E^{A,n}_{\textbf{u}}}{dt}, \textbf{v}_h) + (\frac{E^{A,n+1}_c- E^{A,n}_c}{dt}, d_h) + \int_{\Omega} \mu(c^n) \bigtriangledown E^{A,n,\theta}_{\textbf{u}} :  \bigtriangledown \textbf{v}_h +\\
 \int_{\Omega} \mu(c^n) E^{A,n,\theta}_{\textbf{u}} \cdot \textbf{v}_h +\int_{\Omega} \tilde{\bigtriangledown} E^{A,n,\theta}_c \cdot \bigtriangledown \textbf{v}_h + \alpha \int_{\Omega} E^{A,n,\theta}_c d_h = \int_{\Omega} (\bigtriangledown \cdot \textbf{v}_h) (E^{I,n,\theta}_p + E^{A,n,\theta}_p)\\
 -\int_{\Omega}  (\bigtriangledown \cdot E^{A,n,\theta}_{\textbf{u}}) q_h - \int_{\Omega} \mu(c^n)  \bigtriangledown E^{I,n,\theta}_{\textbf{u}} :  \bigtriangledown \textbf{v}_h - \sigma \int_{\Omega} E^{I,n,\theta}_{\textbf{u}} \cdot \textbf{v}_h - \int_{\Omega} \tilde{\bigtriangledown} E^{I,n,\theta}_c \cdot \bigtriangledown d_h- \\
 \alpha \int_{\Omega} E^{I,n,\theta}_c d_h - \int_{\Omega} d_h \textbf{u}^{n} \cdot \bigtriangledown E^{I,n,\theta}_c - \int_{\Omega} d_h \textbf{u}^{n} \cdot \bigtriangledown E^{A,n,\theta}_c - \int_{\Omega} d_h E^{I,n}_{\textbf{u}} \cdot \bigtriangledown c_h^{n,\theta}- \\
 \int_{\Omega} d_h E^{A,n}_{\textbf{u}} \cdot \bigtriangledown c_h^{n,\theta} -c(\textbf{u}^{n},E^{I,n,\theta}_{\textbf{u}}, \textbf{v}_h)- c(\textbf{u}^{n},E^{A,n,\theta}_{\textbf{u}}, \textbf{v}_h) - c(E^{I,n}_{\textbf{u}}, \textbf{u}^{n, \theta}_h,\textbf{v}_h) \\
 - c(E^{A,n}_{\textbf{u}}, \textbf{u}^{n, \theta}_h,\textbf{v}_h) + \int_{\Omega} \mu(c^n) E^{A,n,\theta}_{\textbf{u}} \cdot \textbf{v}_h -I_1- I_2 -I_3 - I_4 - (\textbf{TE}^{n,\theta}, \textbf{V}_h) \hspace{5mm} 
\end{multline}
Applying various properties of the projection operators we have the final expression of $I_1$ above.

\begin{equation}
\begin{split}
I_1 &=\sum_{k=1}^{n_{el}}(\tau_k'(M\partial_t (\textbf{U}^n-\textbf{U}^n_h)+ \mathcal{L}(\textbf{u}^{n} ;\textbf{U}^{n,\theta})-\mathcal{L}(\textbf{u}^{n}_h ;\textbf{U}^{n,\theta}_h)),-\mathcal{L}^* (\textbf{u}_h;\textbf{V}_h))_{\Omega_k} \\
& = \sum_{k=1}^{n_{el}} [( \tau_{1k}' I_{d \times d} \{\rho \partial_t E^{I,n}_{\textbf{u}}+ \rho E^{I,n}_{\textbf{u}} \cdot \bigtriangledown \textbf{u}^{n,\theta} + \rho \textbf{u}_h^{n} \cdot \bigtriangledown E^{I,n,\theta}_{\textbf{u}} - \mu(c^n) \Delta E^{I,n,\theta}_{\textbf{u}} \\
& \quad + \bigtriangledown E^{I,n,\theta}_p \}, \rho(\textbf{u}_h \cdot \bigtriangledown) \textbf{v}_h + \mu(c) \Delta \textbf{u}_h + \bigtriangledown p_h)_{\Omega_k} + ( \tau_{1k}' I_{d \times d} \{\rho \partial_t E^{A,n}_{\textbf{u}}+ \rho \\
& \quad E^{A,n}_{\textbf{u}} \cdot \bigtriangledown \textbf{u}^{n,\theta} + \rho \textbf{u}_h^{n} \cdot \bigtriangledown E^{A,n,\theta}_{\textbf{u}} - \mu(c^n) \Delta E^{A,n,\theta}_{\textbf{u}} + \bigtriangledown E^{A,n,\theta}_p \}, \rho(\textbf{u}_h \cdot \bigtriangledown) \textbf{v}_h \\
& \quad + \mu(c) \Delta \textbf{u}_h + \bigtriangledown p_h)_{\Omega_k} + (\tau_{2k}' \bigtriangledown \cdot (E^{I,n,\theta}_{\textbf{u}}+E_{\textbf{u}}^{A,n,\theta}), \bigtriangledown \cdot \textbf{v}_h)_{\Omega_k}  + (\tau_{3k}' \{ \partial_t E^{I,n}_{c} \\ 
 & \quad - \bigtriangledown \cdot \tilde{\bigtriangledown} E^{I,n,\theta}_{c}+( E_{\textbf{u}}^{I,n} \cdot \bigtriangledown )c^{n,\theta} + (\textbf{u}_h^{n} \cdot \bigtriangledown) E^{I,n,\theta}_{c}+\alpha E_{c}^{I,n,\theta} \}, \bigtriangledown \cdot \tilde{\bigtriangledown} d_h \\
& \quad + \textbf{u}_h \cdot \bigtriangledown d_h - \alpha d_h )_{\Omega_k} + (\tau_3' \{ \partial_t E^{A,n}_{c} -\bigtriangledown \cdot \tilde{\bigtriangledown} E^{A,n,\theta}_{c}+( E_{\textbf{u}}^{A,n} \cdot \bigtriangledown )c^{n,\theta}   \\
& \quad + (\textbf{u}_h^{n}\cdot \bigtriangledown) E^{A,n,\theta}_{c}+\alpha E_{c}^{A,n,\theta} \},\bigtriangledown \cdot \tilde{\bigtriangledown} d_h + \textbf{u}_h \cdot \bigtriangledown d_h - \alpha d_h )_{\Omega_k}]\\
& = I_1^1 +I_1^2+I_1^3+I_1^4+ I_1^5 \hspace{2mm} (say)
\end{split} 
\end{equation} 
where $I_1^i$ for $i=1,2,...,5$ are five terms of $I_1$ which we will discuss in the later part of the proof and since $(1- \tau_2^{-1}\tau_2')=0$ the next term will take the following form 
\begin{equation}
\begin{split}
I_2 & = \sum_{k=1}^{n_{el}}((I-\tau_k^{-1}\tau_k')(M \partial_t(\textbf{U}^{n}-\textbf{U}^{n}_h) + \mathcal{L}(\textbf{u}^{n} ;\textbf{U}^{n,\theta})-\mathcal{L}(\textbf{u}^{n}_h ;\textbf{U}^{n,\theta}_h)), -\textbf{V}_h)_{\Omega_k} \\
\end{split}
\end{equation}
\begin{equation}
\begin{split}
& = \sum_{k=1}^{n_{el}} [(\frac{\rho \tau_{1k}}{dt+\rho \tau_{1k}}I_{d \times d} \{\rho \partial_t E^{I,n}_{\textbf{u}}+ \rho E^{I,n}_{\textbf{u}} \cdot \bigtriangledown \textbf{u}^{n,\theta} + \rho \textbf{u}_h^{n} \cdot \bigtriangledown E^{I,n,\theta}_{\textbf{u}} - \mu(c^n) \\
& \quad  \Delta E^{I,n,\theta}_{\textbf{u}}+ \bigtriangledown E^{I,n,\theta}_p \}, - \textbf{v}_h)_{\Omega_k} + (\frac{\rho \tau_{1k}}{dt+\rho \tau_{1k}}I_{d \times d} \{\rho \partial_t E^{A,n}_{\textbf{u}} + \rho E^{A,n}_{\textbf{u}} \cdot \bigtriangledown \textbf{u}^{n,\theta}\\
& \quad  + \rho \textbf{u}_h^{n} \cdot \bigtriangledown E^{A,n,\theta}_{\textbf{u}} - \mu(c^n) \Delta E^{A,n,\theta}_{\textbf{u}}+ \bigtriangledown E^{A,n,\theta}_p \}, - \textbf{v}_h)_{\Omega_k} + (\frac{\tau_{3k}}{dt+\tau_{3k}}\{ \partial_t E^{I,n}_{c} \\
& \quad - \bigtriangledown \cdot \tilde{\bigtriangledown} E^{I,n,\theta}_{c}+( E_{\textbf{u}}^{I,n} \cdot \bigtriangledown )c^{n,\theta} + (\textbf{u}_h^{n} \cdot \bigtriangledown) E^{I,n,\theta}_{c} +\alpha E^{I,n,\theta}_{c}\}, -d_h)_{\Omega_k} +\\
& \quad (\frac{\tau_{3k}}{dt+\tau_{3k}}\{ \partial_t E^{A,n}_{c}- \bigtriangledown \cdot \tilde{\bigtriangledown} E^{A,n,\theta}_{c}+( E_{\textbf{u}}^{A,n} \cdot \bigtriangledown )c^{n,\theta} + (\textbf{u}_h^{n} \cdot \bigtriangledown) E^{A,n,\theta}_{c} + \\
& \quad \alpha E^{A,n,\theta}_{c}\}, -d_h)_{\Omega_k}] \\
& = I_2^1 +I_2^2+I_2^3+ I_2^4 \hspace{2mm} (say)
\end{split}
\end{equation} 
The next terms of $LHS$ in (36) are as follows:
\begin{equation}
\begin{split}
I_3 & =\sum_{k=1}^{n_{el}} (\tau_k^{-1}\tau_k' \textbf{d}, \textbf{V}_h)_{\Omega_k} \\
&= \sum_{k=1}^{n_{el}} (\{\frac{dt}{dt+ \rho \tau_{1k}} I_{d \times d} \textbf{d}, \textbf{v}_h)_{\Omega_k} + (\frac{dt}{dt+\tau_{3k}} d_4, d_{h})_{\Omega_k} \} \hspace{30mm} \\
\end{split}
\end{equation}
and 
\begin{equation}
\begin{split}
I_4 & = \sum_{k=1}^{n_{el}}(\tau_k' \textbf{d},-\mathcal{L}^*(\textbf{u}_h; \textbf{V}_h))_{\Omega_k} \\
& =  \sum_{k=1}^{n_{el}}\{ (\tau_{1k}' I_{d \times d} \textbf{d},\rho (\textbf{u}_h \cdot \bigtriangledown) \textbf{v}_h +\mu(c) \Delta \textbf{v}_h  + \bigtriangledown q_h)_{\Omega_k}+ (\tau_{3k}' d_3, \bigtriangledown \cdot \tilde{\bigtriangledown} d_h +\\
& \quad \textbf{u}_h \cdot \bigtriangledown d_h - \alpha d_h )_{\Omega_k}\} \\
\end{split}
\end{equation}
Now we will treat each term separately to find out the estimate. Our procedure contains finding two bounds: one is lower bound of $LHS$ and the other one is upper bound for the terms in $RHS$ and combining those bounds in the equation (36) we will finally obtain the required estimate. Before further proceeding let us mention an important consideration: since the above equation holds for all $\textbf{V}_h \in V_s^h \times V_s^h \times Q_s^h \times V_s^h$, therefore in each term we replace $v_{1h},v_{2h},q_h,d_h$ by $E^{A,n,\theta}_{u1},E^{A,n,\theta}_{u2},E^{A,n,\theta}_{p},E^{A,n,\theta}_{c}$ respectively as these auxiliary part of the errors belonging to their respective finite element spaces. From now onwards we will start derivation of each expression after considering the replacements directly.\vspace{2mm}\\
Applying the result obtained in (32) on the first term of $LHS$ and taking out the $infimum$ of the coefficients of the remaining terms we can easily see that
\begin{multline}
\frac{\rho}{2 dt} (\|E^{A,n+1}_{\textbf{u}}\|^2-\|E^{A,n}_{\textbf{u}}\|^2)+ \frac{1}{2 dt}(\|E^{A,n+1}_{c}\|^2-\|E^{A,n}_{c}\|^2)+ \mu_l \| E^{A,n,\theta}_{\textbf{u}}\|_1^2  \\
+D_l \mid E^{A,n,\theta}_c \mid_1^2  +  \alpha \|E^{A,n,\theta}_{c}\|^2  \leq LHS = RHS\\
\end{multline}
where $D_l$= min $\{ \underset{\Omega}{inf} D_1, \underset{\Omega}{inf} D_2  \}$. \\
Now we will find upper bounds of each of the terms in the $RHS$ of the equation (37). We usually use $Cauchy-Schwarz$ and $Young's$ inequality to reach at the desired bounds. We have already estimated the bounds of few terms on $RHS$ in $\cite{RefN}$. Therefore we only mention the results here and the estimation of the remaining terms are shown later in details.
\begin{equation}
\begin{split}
\int_{\Omega}(\bigtriangledown \cdot E^{A,n,\theta}_\textbf{u})E^{I,n,\theta}_{p} 
& \leq  \epsilon_1 C^2 h^2 (\frac{1+\theta}{2}\| p^{n+1}\|_1+ \frac{1-\theta}{2}\| p^n \|_1)^2 +\\
& \quad \frac{1}{2 \epsilon_1}\mid E^{A,n,\theta}_{\textbf{u}} \mid_1^2\\
-\int_{\Omega} \mu(c^n)\bigtriangledown E^{I,n,\theta}_{\textbf{u}} : \bigtriangledown E^{A,n,\theta}_\textbf{u} & \leq \epsilon_2 \mu_u  C^2 h^2(\frac{1+\theta}{2}\| \textbf{u}^{n+1}\|_2+ \frac{1-\theta}{2} \| \textbf{u}^n\|_2)^2 +\\
& \quad  \frac{\mu_u}{2 \epsilon_2} \mid E^{A,n,\theta}_{\textbf{u}} \mid_1^2 \\
- \int_{\Omega} \tilde{\bigtriangledown}E^{I,n,\theta}_{c} \cdot \bigtriangledown E^{A,n,\theta}_c & \leq    \frac{D_m \epsilon_3}{2} C^2 h^2 (\frac{1+\theta}{2} \| c^{n+1} \|_2+\frac{1-\theta}{2}\| c^n \|_2)^2 + \\
& \quad \frac{D_m}{2 \epsilon_3}(\|\frac{\partial E^{A,n,\theta}_{c}}{\partial x}\|^2 + \|\frac{\partial E^{A,n,\theta}_{c}}{\partial y}\|^2)\\
- \sigma \int_{\Omega} E^{I,n,\theta}_{\textbf{u}} \cdot E^{A,n,\theta}_{\textbf{u}} & \leq \frac{\epsilon_4}{2}\sigma h^4 (\frac{1+\theta}{2} \|\textbf{u}^{n+1}\|_2 + \frac{1-\theta}{2} \|\textbf{u}^{n}\|_2)^2 + \\
& \quad \frac{\sigma}{ 2\epsilon_4} \|E^{A,n,\theta}_{\textbf{u}}\|^2 \\
-\alpha \int_{\Omega} E^{I,n,\theta}_c E^{A,n,\theta}_c & \leq \frac{\epsilon_5}{2} \alpha h^4 (\frac{1+\theta}{2} \|c^{n+1}\|_2 + \frac{1-\theta}{2} \|c^n\|_2)^2 + \\
& \quad \frac{\alpha}{2 \epsilon_5} \|E^{A,n,\theta}_c \|^2
\end{split}
\end{equation}  
where $D_m$= max $\{ \underset{\Omega}{sup} D_1, \underset{\Omega}{sup} D_2 \}$ \\
Now applying $Poincare$  inequality in the following we have:
\begin{equation}
\begin{split}
\int_{\Omega} \mu(c^n)\{(E^{A,n,\theta}_{u1})^2+(E^{A,n,\theta}_{u2})^2 \} & \leq \mu_u (\|E^{A,n,\theta}_{u1}\|^2+\|E^{A,n,\theta}_{u2}\|^2)\\
& \leq \mu_u C_P (\mid E^{A,n,\theta}_{u1}\mid_1^2+\mid E^{A,n,\theta}_{u2}\mid_1^2)
\end{split}
\end{equation}  
where $C_P$ is the $Poincare$ constant.  The next term is estimated following $\cite{RefN}$
\begin{equation}
\begin{split}
-\int_{\Omega}E^{A,n,\theta}_c \textbf{u}^{n} \cdot \bigtriangledown E^{A,n,\theta}_{c} &  \leq 2 \bar{C}_1^2 (\| \frac{\partial E^{A,n,\theta}_c}{\partial x} \|^2 + \| \frac{\partial E^{A,n,\theta}_c}{\partial y} \|^2) \\
-\int_{\Omega}E^{A,n,\theta}_c \textbf{u}^{n} \cdot \bigtriangledown E^{I,n,\theta}_{c} &  \leq \bar{C}_1^2 C^2 h^2 \epsilon_6 (\frac{1+\theta}{2}\| c^{n+1} \|_2 + \frac{1-\theta}{2} \| c^n \|_2)^2+ \\
& \quad \frac{\bar{C}_1^2}{\epsilon_6} \| E^{A,n,\theta}_c\|^2 
\end{split}
\end{equation}
Now we estimate the trilinear term $c(\cdot,\cdot,\cdot)$ using it's properties \textbf{(a)} and \textbf{(b)} given in section 2. Let us start the estimation with the first trilinear term on $RHS$ of  (37) as follows:
\begin{equation}
\begin{split}
-c(\textbf{u}^{n},E^{I,n,\theta}_{\textbf{u}}, E^{A,n,\theta}_{\textbf{u}}) & \leq C \|\textbf{u}^{n}\|_2 \|E^{I,n,\theta}_{\textbf{u}}\|_1 \|E^{A,n,\theta}_{\textbf{u}}\| \\
& \leq  \frac{C_2 \epsilon_7}{2} \|E^{I,n,\theta}_{\textbf{u}}\|_1^2 + \frac{C_2}{2 \epsilon_7} \|E^{A,n,\theta}_{\textbf{u}}\|^2 \\
& \leq  \frac{C_2 \epsilon_7}{2} (\frac{1+\theta}{2}\|E^{I,n+1}_{\textbf{u}}\|_1 + \frac{1-\theta}{2}\|E^{I,n}_{\textbf{u}}\|_1)^2 + \frac{C_2}{2 \epsilon_7} \|E^{A,n,\theta}_{\textbf{u}}\|^2 \\
& \leq  \frac{ \epsilon_7}{2} C_2 h^2 (\frac{1+\theta}{2} \|\textbf{u}^{n+1}\|_2 + \frac{1-\theta}{2} \|\textbf{u}^{n}\|_2)^2 + \frac{C_2}{2 \epsilon_7} \|E^{A,n,\theta}_{\textbf{u}}\|^2
\end{split}
\end{equation}
By the property \textbf{(a)} of trilinear case for both the linear and non-linear cases:
\begin{equation}
-c(\textbf{u}^{n},E^{A,n,\theta}_{\textbf{u}}, E^{A,n,\theta}_{\textbf{u}})=0 
\end{equation}
The next term,
\begin{equation}
\begin{split}
-c(E^{I,n}_{\textbf{u}},\textbf{u}^{n,\theta}_h, E^{A,n,\theta}_{\textbf{u}}) & = c(E^{I,n}_{\textbf{u}},E^{I,n,\theta}_{\textbf{u}}, E^{A,n,\theta}_{\textbf{u}})+ c(E^{I,n}_{\textbf{u}},E^{A,n,\theta}_{\textbf{u}}, E^{A,n,\theta}_{\textbf{u}}) -\\
& \quad c(E^{I,n}_{\textbf{u}},\textbf{u}^{n,\theta}, E^{A,n,\theta}_{\textbf{u}}) \\
& \leq C \|E^{I,n}_{\textbf{u}}\| \|E^{I,n,\theta}_{\textbf{u}}\|_1 \|E^{A,n,\theta}_{\textbf{u}}\|_1 +C \|E^{I,n,\theta}_{\textbf{u}}\| \|\textbf{u}^{n,\theta}\|_2 \|E^{A,n,\theta}_{\textbf{u}}\|_1 \\
& \leq h^4 \frac{C_2 \epsilon_8 }{2} \|\textbf{u}^{n}\| (\frac{1+\theta}{2} \|\textbf{u}^{n+1}\|_2 + \frac{1-\theta}{2} \|\textbf{u}^{n}\|_2)^2 +\frac{C_2}{ \epsilon_8} \|E^{A,n,\theta}_{\textbf{u}}\|_1^2 \\
& \quad +h^4 \frac{C_2 \epsilon_8}{2 } (\frac{1+\theta}{2} \|\textbf{u}^{n+1}\|_2 + \frac{1-\theta}{2} \|\textbf{u}^{n}\|_2)^2 
\end{split}
\end{equation}
and 
\begin{equation}
\begin{split}
-c(E^{A,n}_{\textbf{u}},\textbf{u}^{n,\theta}_h, E^{A,n,\theta}_{\textbf{u}}) & = c(E^{A,n}_{\textbf{u}},E^{I,n,\theta}_{\textbf{u}}, E^{A,n,\theta}_{\textbf{u}}) + c(E^{A,n}_{\textbf{u}},E^{A,n,\theta}_{\textbf{u}}, E^{A,n,\theta}_{\textbf{u}})  \\
& \quad - c(E^{A,n}_{\textbf{u}},\textbf{u}^{n,\theta}, E^{A,n,\theta}_{\textbf{u}}) \\
& \leq C  \|E^{A,n,\theta}_{\textbf{u}}\| \|E^{A,n,\theta}_{\textbf{u}}\|_1 \{ \|E^{I,n,\theta}_{\textbf{u}}\|_2  + \|\textbf{u}^{n,\theta}\|_2  \}\\
& \leq C \|E^{A,n,\theta}_{\textbf{u}}\|_1^2 (\frac{1+\theta}{2}\|E^{I,n+1}_{\textbf{u}}\|_2 + \frac{1-\theta}{2}\|E^{I,n}_{\textbf{u}}\|_2) + \\
& \quad C_2 \|E^{A,n,\theta}_{\textbf{u}}\|_1^2\\
& \leq C \|E^{A,n,\theta}_{\textbf{u}}\|_1^2 \{ (\frac{1+\theta}{2}\|\textbf{u}^{n+1}\|_2 + \frac{1-\theta}{2}\|\textbf{u}^n\|_2)+ C_2 \}\\
& \leq C_2' \|E^{A,n,\theta}_{\textbf{u}}\|_1^2
\end{split}
\end{equation}
Now we will find bounds for each remaining term of $I_1$ separately. Before going to further calculations let us mention an important observation:
By the virtue of the choices of the finite element spaces $V^h$ and $Q^h$, we can clearly say that over each element sub-domain $\Omega_k$ every function belonging to that spaces and their first and second order derivatives all are bounded functions. We can always find positive finite real numbers to bound each of the functions over element sub-domain. We will use this fact for several times further. \vspace{1mm}\\
Let us start with $I_1^1$
\begin{equation}
\begin{split}
I_1^1 & = \sum_{k=1}^{n_{el}} (\rho \tau_{1k}' I_{d \times d}  \partial_t E^{I,n}_{\textbf{u}},  \rho(\textbf{u}_h \cdot \bigtriangledown) E^{A,n,\theta}_{\textbf{u}} + \mu(c) \Delta E^{A,n,\theta}_{\textbf{u}} + \bigtriangledown E^{A,n,\theta}_{p})_{\Omega_k} +   \\
& \quad \sum_{k=1}^{n_{el}} ( \tau_{1k}' I_{d \times d} \{ \rho E^{I,n,\epsilon \theta}_{\textbf{u}} \cdot \bigtriangledown \textbf{u}^{n,\theta} + \rho \textbf{u}_h^{n,\theta} \cdot \bigtriangledown E^{I,n,\theta}_{\textbf{u}} - \mu(c^n) \Delta E^{I,n,\theta}_{\textbf{u}}  + \bigtriangledown E^{I,n,\theta}_p \},  \\
& \quad  \rho(\textbf{u}_h \cdot \bigtriangledown) E^{A,n,\theta}_{\textbf{u}} + \mu(c) \Delta E^{A,n,\theta}_{\textbf{u}} + \bigtriangledown E^{A,n,\theta}_{p})_{\Omega_k} \\
& = I_1^{11}+ I_1^{12} \hspace{2mm} (say)
\end{split}
\end{equation}
Now we present the estimations of these two terms separately in details. According to the above observation we can find bounds on each of the terms $\textbf{u}_h, E^{A,n,\theta}_{\textbf{u}},  E^{A,n,\theta}_p$ and their first and second order derivatives over each sub-domain $\Omega_k$.  Applying these bounds in the following we will have
\begin{equation}
\begin{split}
 I_1^{11} & = \int_{\Omega'} \rho \tau_{1k}'I_{d \times d} \frac{E^{I,n+1}_{\textbf{u}}-E^{I,n}_{\textbf{u}}}{dt} \{ \rho(\textbf{u}_h \cdot \bigtriangledown) E^{A,n,\theta}_{\textbf{u}} + \mu(c) \Delta E^{A,n,\theta}_{\textbf{u}} + \bigtriangledown E^{A,n,\theta}_{p}\} \\
 & \leq \rho \frac{\mid \tau_{1}'\mid}{dt} \{\sum_{k=1}^{n_{el}} D_{B_{1k}} \} (\|E^{I,n+1}_{\textbf{u}}\|+ \|E^{I,n}_{\textbf{u}}\|)\\
& \leq \rho Ch^2 \frac{\mid \tau_{1}\mid}{\mid dt+ \rho \tau_1 \mid}\{\sum_{k=1}^{n_{el}} D_{B_{1k}}\} (\|\textbf{u}^{n+1}\|_2+\|\textbf{u}^n\|_2)  \\
& \leq  h^2 \frac{\rho C C_{\tau_1}}{(T_0-\rho C_{\tau_{1}})} \{\sum_{k=1}^{n_{el}} D_{B_{1k}}\}(\|\textbf{u}^{n+1}\|_2+\|\textbf{u}^n\|_2)  \\
\end{split}
\end{equation}
where the constant $D_{B_{1k}}$ is obtained after imposing bounds on the above bracketed terms over each sub-domain $\Omega_k$. $C_{\tau_1}$and $T$ are upper bounds on respectively $\tau_1$ and $dt$. Since $dt$ is a non-zero positive real number, let $T_0$ is lower bound on $dt$. In order to make $(T_0-\rho C_{\tau_{1}})$ positive we have to take $h$ very small.\vspace{1mm}\\
Now the estimation of the second term is as follows:
\begin{equation}
\begin{split}
I_1^{12} & = \sum_{k=1}^{n_{el}} ( \tau_{1k}' I_{d \times d} \{ \rho E^{I,n}_{\textbf{u}} \cdot \bigtriangledown \textbf{u}^{n,\theta} + \rho \textbf{u}^{n} \cdot \bigtriangledown E^{I,n,\theta}_{\textbf{u}} - \rho E^{I,n}_{\textbf{u}} \cdot \bigtriangledown E^{I,n,\theta}_{\textbf{u}}- \\
& \quad  \rho E^{A,n}_{\textbf{u}} \cdot \bigtriangledown E^{I,n,\theta}_{\textbf{u}}- \mu(c^n) \Delta E^{I,n,\theta}_{\textbf{u}}  + \bigtriangledown E^{I,n,\theta}_p \}, \rho(\textbf{u}_h \cdot \bigtriangledown) E^{A,n,\theta}_{\textbf{u}} +  \\
& \quad \mu(c) \Delta E^{A,n,\theta}_{\textbf{u}} + \bigtriangledown E^{A,n,\theta}_{p})_{\Omega_k} \\
& \leq \mid \tau_{1k}' \mid [ \sum_{k=1}^{n_{el}} D_{B_{1k}} \{ \sum_{i=1}^d (\rho \|E^{I,n}_{ui} \| \| \frac{\partial u_i^{n,\theta}}{\partial x_i} \| + \rho \|u_i^{n}\| \|\frac{\partial E^{I,n,\theta}_{ui}}{\partial x_i}\| + \| E^{I,n}_{ui}\|  \\
& \quad \rho \| \frac{\partial E^{I,n,\theta}_{ui}}{\partial x_i}\| + \rho\|E^{A,n}_{ui}\|_k \| \frac{\partial E^{I,n,\theta}_{ui}}{\partial x_i}\|+ \mu(c) \| \Delta E^{I,n,\theta}_{ui}\| + \| \frac{\partial E^{I,n,\theta}_p}{\partial x_i}\| )\}]\\
 & \leq \frac{\mid \tau_{1}\mid T}{(T_0- \rho C_{\tau_1})} C (\sum_{k=1}^{n_{el}} D_{B_{1k}}) [\sum_{i=1}^d \{\rho h^2 \|\frac{\partial u_i^{n, \theta}}{\partial x_i}\| + \rho h \|u_i^{n,\epsilon\theta}\| + \rho h B^i_{1k} + \mu_u  \\
& \quad  +h^3 (\frac{1+\theta}{2} \|u_i^{n+1}\|_2 + \frac{1-\theta}{2} \|u_i^n\|_2)\} (\frac{1+\theta}{2} \|u_i^{n+1}\|_2 + \frac{1-\theta}{2} \|u_i^n\|_2) + \\
& \quad h (\frac{1+\theta}{2} \|p^{n+1}\|_1 + \frac{1-\theta}{2} \|p^n\|_1) ]
\end{split}
\end{equation} 
where $B^i_{1k}$ are bounds on $E^{A,n,\theta}_{ui}$ for $i=1,...,d$. Now the next term $I_1^2$ can be estimated by dividing it into two terms $I_1^{21}$ and $ I_1^{22}$ as above. Therefore we directly start the estimation here with the first term of $I_1^2$ denoting that by $I_1^{21}$.
\begin{equation}
\begin{split}
I_1^{21} & =  \int_{\Omega'} \rho \tau_{1k}'I_{d \times d} \frac{E^{A,n+1}_{\textbf{u}}-E^{A,n}_{\textbf{u}}}{dt} \{ \rho(\textbf{u}_h \cdot \bigtriangledown) E^{A,n,\theta}_{\textbf{u}} + \mu(c) \Delta E^{A,n,\theta}_{\textbf{u}} + \bigtriangledown E^{A,n,\theta}_{p}\} \\
& \leq  \frac{\rho T C_{\tau_1}}{dt (T_0- \rho C_{\tau_1})} (\sum_{k=1}^{n_{el}} D_{B_{1k}}) \{ \|E^{A,n+1}_{\textbf{u}} \|^2 - \| E^{A,n}_{\textbf{u}}\|^2 \}
\end{split}
\end{equation}
and
\begin{equation}
\begin{split}
I_1^{22} & =  \sum_{k=1}^{n_{el}} ( \tau_{1k}' I_{d \times d} \{ \rho E^{A,n}_{\textbf{u}} \cdot \bigtriangledown \textbf{u}^{n,\theta} + \rho \textbf{u}^{n} \cdot \bigtriangledown E^{A,n}_{\textbf{u}} - \rho E^{I,n}_{\textbf{u}} \cdot \bigtriangledown E^{A,n,\theta}_{\textbf{u}}-\rho E^{A,n}_{\textbf{u}} \cdot \bigtriangledown  \\
& \quad E^{A,n,\theta}_{\textbf{u}} - \mu(c^n) \Delta E^{A,n,\theta}_{\textbf{u}}  + \bigtriangledown E^{A,n,\theta}_p \}, \rho(\textbf{u}_h \cdot \bigtriangledown) E^{A,n,\theta}_{\textbf{u}} +  \mu(c) \Delta E^{A,n,\theta}_{\textbf{u}} + \\
& \quad \bigtriangledown E^{A,n,\theta}_{p})_{\Omega_k} \\
& \leq \mid \tau_1' \mid  [ \sum_{k=1}^{n_{el}} D_{B_{1k}} \{ \sum_{i=1}^d (\rho  \|E^{A,n}_{ui}\|_k \| \frac{\partial u_i^{n}}{\partial x_i} \| + \rho \|u_i^{n,\theta} \| \| \frac{\partial E^{A,n,\theta}_{ui}}{\partial x_i}\|_k + \rho  \| E^{I,n}_{ui}\| \\
& \quad  \| \frac{\partial E^{A,n,\theta}_{ui}}{\partial x_i}\|_k +  \| E^{A,n}_{ui}\|_k \| \frac{\partial E^{A,n,\theta}_{ui}}{\partial x_i}\|_k + \mu_u \| \Delta E^{A,n,\theta}_{ui} \|_k + \| \frac{\partial E^{A,n,\theta}_p}{\partial x_i} \|_k) \} ] 
\end{split}
\end{equation}
Applying bounds on the functions belonging to $V^h$ and $Q^h$ on the above equation and denoting that bound by $\bar{D}_{B_{1k}}$ we have 
\begin{equation}
I_1^{22} \leq \frac{\mid \tau_1 \mid T}{(T_0- \rho C_{\tau_1})} (\sum_{k=1}^{n_{el}} D_{B_{1k}} \bar{D}_{B_{1k}})
\end{equation}
Now expanding out the next term of $I_1$ we can proceed to estimate that in the following way:
\begin{equation}
\begin{split}
I_1^3 & =\sum_{k=1}^{n_{el}} \int_{\Omega_k}\tau_2 \{ ( \bigtriangledown \cdot E^{I,n,\theta}_{\textbf{u}}) \cdot ( \bigtriangledown \cdot E^{A,n,\theta}_{\textbf{u}}) + ( \bigtriangledown \cdot E^{A,n,\theta}_{\textbf{u}})^2 \} \\
& \leq C_{\tau_2}\{ \sum_{i=1}^d (\| \frac{\partial E^{I,n,\theta}_{ui}}{\partial x_i}\|^2 + C_1 \| \frac{\partial E^{A,n,\theta}_{ui}}{\partial x_i}\|^2) \}\\
& \leq C_{\tau_2} \{ h^2 \sum_{i=1}^d (\frac{1+\theta}{2}\|u_i^{n+1}\|_2+ \frac{1-\theta}{2}\|u_i^n\|_2)^2+ \mid E^{A,n,\theta}_{\textbf{u}}\mid_1^2 \}
\end{split}
\end{equation}
where $C_{\tau_2}$ is the maximum numerical value of $\tau_2$ over $\Omega$. Now the remaining terms of $I_1$ associated with the variable $c$ representing concentration are estimated as follows:
\begin{equation}
\begin{split}
I_1^4 & = \sum_{k=1}^{n_{el}} \tau_{3k}'(\partial_t E^{I,n}_c, \bigtriangledown \cdot\tilde{\bigtriangledown} E^{A,n,\theta}_c+\textbf{u}_h \cdot \bigtriangledown E^{A,n,\theta}_c -\alpha E^{A,n,\theta}_c)_{\Omega_k}  +\\
& \quad \sum_{k=1}^{n_{el}} \tau_{3k}'(-\bigtriangledown \cdot\tilde{\bigtriangledown} E^{I,n,\theta}_c + (E^{I,n}_{\textbf{u}} \cdot \bigtriangledown)c^{n,\theta}+ (\textbf{u}_h^{n} \cdot \bigtriangledown)E^{I,n,\theta}_c + \alpha E^{I,n,\theta}_c ,\\
& \quad \bigtriangledown \cdot\tilde{\bigtriangledown} E^{A,n,\theta}_c+\textbf{u}_h \cdot \bigtriangledown E^{A,n,\theta}_c -\alpha E^{A,n,\theta}_c)_{\Omega_k} \\
& \leq C h^2 \mid \tau_3' \mid \{ \sum_{k=1}^{n_{el}} D_{B_{3k}} \} (\|c^{n+1}\|_2+\|c^n\|_2) + \mid \tau_3' \mid [\sum_{k=1}^{n_{el}} D_{B_{3k}} \{ \sum_{i=1}^d (\|E^{I,n}_{ui}\|  \\
& \quad \|\frac{\partial c^{n,\theta}}{\partial x_i}\| + D_{im} \| \frac{\partial^2 E^{I,n,\theta}_c}{\partial x_i^2}\|+ \bar{D}_{im} \|\frac{\partial E^{I,n,\theta}_c}{\partial x_i}\|+ (\|u_i^{n}\|+\|E^{I,n}_{ui}\|+ \|E^{A,n}_{ui}\|)\\
& \quad \|\frac{\partial E^{I,n,\theta}_c}{\partial x_i}\|+\mid \alpha \mid \|E^{I,n,\theta}_c\| )\}] \\
& \leq h^2 \frac{C C_{\tau_3}}{(T_0- C_{\tau_3})} \{ \sum_{k=1}^{n_{el}} D_{B_{3k}} \} (\|c^{n+1}\|_2+\|c^n\|_2) + \frac{C \mid \tau_3 \mid}{(T_0- C_{\tau_3})}  [\sum_{k=1}^{n_{el}} D_{B_{3k}} \sum_{i=1}^d \{h^2 \\
& \quad (\frac{1+\theta}{2}\|u_i^{n+1}\|_2+\frac{1-\theta}{2}\|u_i^n\|_2)\|\frac{\partial c^{n,\theta}}{\partial x_i}\| + (D_{im}+h \bar{D}_{im}+ \mid \alpha\mid  h^2+ h\|u_i^{n,\theta}\| \\
& \quad +hB_{1k}^i+ h^3 (\frac{1+\theta}{2}\|u_i^{n+1}\|_2+\frac{1-\theta}{2}\|u_i^n\|_2))(\frac{1+\theta}{2}\|c^{n+1}\|_2+\frac{1-\theta}{2} \|c^n\|_2)\}]
\end{split}
\end{equation}
Let $D_{B_{3k}}$ be the summation of the bounds imposed on the elements $\frac{\partial^2 E^{A,n,\theta}_c}{\partial x_i^2}$, $\frac{\partial E^{A,n,\theta}_c}{\partial x_i}$, $ E^{A,n,\theta}_c$ belonging to finite element space $V^h$ and $D_{im}, \bar{D}_{im}$ be the $supremum$ of $D_i$ and $\frac{\partial D_i}{\partial x_i}$ respectively over each sub-domain $\Omega_k$ for $i=1,...,d$. Now the estimation of  the last term $I_1^5$ follows the same way as above and considering $\bar{D}_{B_{3k}}$ as an expression to denote the estimated result briefly the derivation of the bound of $I_1^5$ is in the following:
\begin{equation}
\begin{split}
I_1^{5} & =\sum_{k=1}^{n_{el}}\tau_{3k}'(\partial_t E^{A,n}_c, \bigtriangledown \cdot\tilde{\bigtriangledown} E^{A,n,\theta}_c+\textbf{u}_h \cdot \bigtriangledown E^{A,n,\theta}_c -\alpha E^{A,n,\theta}_c)_{\Omega_k}  +\\
& \quad \sum_{k=1}^{n_{el}} \tau_{3k}'(-\bigtriangledown \cdot\tilde{\bigtriangledown} E^{A,n,\theta}_c + (E^{A,n}_{\textbf{u}} \cdot \bigtriangledown)c^{n,\theta}+ (\textbf{u}_h^{n} \cdot \bigtriangledown)E^{A,n,\theta}_c + \alpha E^{A,n,\theta}_c ,\\
& \quad \bigtriangledown \cdot\tilde{\bigtriangledown} E^{A,n,\theta}_c+\textbf{u}_h \cdot \bigtriangledown E^{A,n,\theta}_c -\alpha E^{A,n,\theta}_c)_{\Omega_k} \\ 
& \leq \frac{T C_{\tau_3}}{dt (T_0-C_{\tau_3})} \{ \sum_{k=1}^{n_{el}} D_{B_{3k}} \} (\|E^{A,n+1}_c \|^2 -\|E^{A,n}_c \|^2) + \\
& \quad \frac{C \mid \tau_3 \mid}{(T_0-C_{\tau_3})} \{\sum_{k=1}^{n_{el}} \bar{D}_{B_{3k}} D_{B_{3k}} \}
\end{split}
\end{equation}
This completes estimation of $I_1$ finally. Now we see that the terms $I_2^i$ (for $i=1,...,4$) are same as that of $I_1$. Therefore we only mention the results here for each of them as follows:
\begin{equation}
\begin{split}
I_2^{1} & = \sum_{k=1}^{n_{el}}( \frac{\rho\tau_{1k}}{dt+\rho\tau_{1k}} I_{d \times d}  \partial_t E^{I,n}_{\textbf{u}}, -E^{A,n,\theta}_{\textbf{u}})_{\Omega_k} +  \sum_{k=1}^{n_{el}} (\frac{\rho\tau_{1k}}{dt+\rho\tau_{1k}} I_{d \times d} \{ \rho E^{I,n}_{\textbf{u}} \cdot \bigtriangledown \textbf{u}^{n,\theta}   \\
& \quad + \rho \textbf{u}^{n} \cdot \bigtriangledown E^{I,n,\theta}_{\textbf{u}}- \rho E^{I,n}_{\textbf{u}} \cdot \bigtriangledown E^{I,n,\theta}_{\textbf{u}}- \rho E^{A,n}_{\textbf{u}} \cdot \bigtriangledown E^{I,n,\theta}_{\textbf{u}} - \mu(c^n) \Delta E^{I,n,\theta}_{\textbf{u}}\\
&\quad  + \bigtriangledown E^{I,n,\theta}_p \},-E^{A,n,\theta}_{\textbf{u}})_{\Omega_k}\\
& \leq  h^2 \frac{\rho C C_{\tau_1}}{(T_0- \rho C_{\tau_{1}})}[ \sum_{i=1}^d \{\sum_{k=1}^{n_{el}} B_{1k}^i \}(\|u_i^{n+1}\|_2+\|u_i^n\|_2)] +  \frac{ \mid \tau_{1}\mid}{(T_0- \rho C_{\tau_1})} \sum_{i=1}^d (\sum_{k=1}^{n_{el}} B_{1k}^i) \\
& \quad  [\{\rho h^2 \|\frac{\partial u_i^{n, \theta}}{\partial x_i}\| +\rho h \|u_i^{n}\| +\rho h B^i_{1k} + \mu_u + h^3 (\frac{1+\theta}{2} \|u_i^{n+1}\|_2 + \frac{1-\theta}{2} \|u_i^n\|_2)\}\\
& \quad   (\frac{1+\theta}{2} \|u_i^{n+1}\|_2 + \frac{1-\theta}{2} \|u_i^n\|_2) +  h (\frac{1+\theta}{2} \|p^{n+1}\|_1 + \frac{1-\theta}{2} \|p^n\|_1)  ]
\end{split}
\end{equation}
and 
\begin{equation}
\begin{split}
I_2^2 & = \sum_{k=1}^{n_{el}}( \frac{\rho\tau_{1k}}{dt+\rho\tau_{1k}} I_{d \times d}  \partial_t E^{A,n}_{\textbf{u}}, -E^{A,n,\theta}_{\textbf{u}})_{\Omega_k} +  \sum_{k=1}^{n_{el}} (\frac{\rho\tau_{1k}}{dt+\rho\tau_{1k}} I_{d \times d} \{ \rho E^{A,n}_{\textbf{u}} \cdot \bigtriangledown \textbf{u}^{n,\theta}   \\
& \quad + \rho \textbf{u}^{n} \cdot \bigtriangledown E^{A,n,\theta}_{\textbf{u}}- \rho E^{I,n}_{\textbf{u}} \cdot \bigtriangledown E^{A,n,\theta}_{\textbf{u}}- \rho E^{A,n}_{\textbf{u}} \cdot \bigtriangledown E^{A,n,\theta}_{\textbf{u}} - \mu(c^n) \Delta E^{A,n,\theta}_{\textbf{u}}\\
& \quad  + \bigtriangledown E^{A,n,\theta}_p \}, -E^{A,n,\theta}_{\textbf{u}})_{\Omega_k}\\
& \leq  \frac{\rho T C_{\tau_1}}{dt (T_0- \rho C_{\tau_1})}[\sum_{i=1}^d (\sum_{k=1}^{n_{el}} B^i_{1k}) ]\{ \|E^{A,n+1}_{\textbf{u}} \|^2 - \| E^{A,n}_{\textbf{u}}\|^2 \} +\\
& \quad \frac{\mid \tau_1 \mid T}{(T_0- \rho C_{\tau_1})} [\sum_{i=1}^d \sum_{k=1}^{n_{el}} \bar{D}_{B_{1k}} B^i_{1k} ]
\end{split}
\end{equation}
From these results it is clear that estimations of the remaining terms of $I_2$ follow the same path as we have done for $I_1^4$ and $I_1^5$. Hence considering $B_{2k}$ as bound for $E^{A,n,\theta}_c$ over each sub-domain $\Omega_k$ we are skipping the repetition in mentioning the similar kind of results, though they will be added up  in the final stage of combining all the results. \vspace{1mm}\\
Now the job is to estimate next part denoted by $I_3$ and $I_4$ of the equation (43) which contain the matrix $\textbf{d}$. Earlier we have mentioned that $d_2$ is $zero$. Let us look at the other three terms explicitly.
\begin{equation}
\begin{split}
\textbf{d}_1 & = \{\sum_{i=1}^{n+1}(\frac{\rho}{dt}\tau_1')^i\} I_{d \times d}[\rho \partial_t(\textbf{u}^n-\textbf{u}_h^n)+ \rho ((\textbf{u}^{n}-\textbf{u}_h^{n}) \cdot \bigtriangledown) \textbf{u}^{n,\theta}+\rho (\textbf{u}_h^{n} \cdot \bigtriangledown) \\
& \quad (\textbf{u}^{n,\theta}-\textbf{u}_{h}^{n,\theta})-\mu(c^n) \Delta(\textbf{u}^{n,\theta}-\textbf{u}_{h}^{n,\theta}) + \bigtriangledown (p^{n,\theta}- p_h^{n,\theta})] \\
& \leq \{\sum_{i=1}^{\infty}(\frac{\rho}{dt}\tau_1')^i\}  I_{d \times d} [\rho \partial_t (E^{I,n}_{\textbf{u}}+E^{A,n}_{\textbf{u}})+ \rho ((E^{I,n}_{\textbf{u}} + E^{A,n}_{\textbf{u}})\cdot \bigtriangledown) \textbf{u}^{n,\theta} + \\
& \quad \rho (\textbf{u}_h^{n} \cdot \bigtriangledown)(E^{I,n,\theta}_{\textbf{u}}+E^{A,n,\theta}_{\textbf{u}})-\mu(c^n) \Delta(E^{I,n,\theta}_{\textbf{u}}+E^{A,n,\theta}_{\textbf{u}})+ \bigtriangledown E^{I,n,\theta}_p \\
& \quad + \bigtriangledown E^{A,n,\theta}_p ] \\
& = \frac{\rho \tau_1'}{dt-\rho \tau_1'} I_{d \times d}[ \{\rho \partial_t E^{I,n}_{\textbf{u}}+  \rho (E^{I,n}_{\textbf{u}} \bigtriangledown)\textbf{u}^{n,\theta}+ \rho (\textbf{u}_h^{n} \cdot \bigtriangledown)E^{I,n,\theta}_{\textbf{u}}-\mu(c^n)  \\
& \quad \Delta E^{I,n,\theta}_{\textbf{u}}+ \bigtriangledown E^{I,n,\theta}_p \} +  \{\rho \partial_t E^{A,n}_{\textbf{u}}+ \rho (E^{A,n}_{\textbf{u}} \bigtriangledown)\textbf{u}^{n,\theta}+\rho (\textbf{u}_h^{n} \cdot \bigtriangledown)E^{A,n,\theta}_{\textbf{u}}  \\
& \quad  -\mu(c^n) \Delta E^{A,n,\theta}_{\textbf{u}}+ \bigtriangledown E^{A,n,\theta}_p \}]
\end{split}
\end{equation}
Since $\frac{\rho \tau_1}{dt+ \rho \tau_1} < 1$, which implies $\frac{\rho \tau_1'}{dt} < 1$ and therefore the series $\sum_{i=1}^{\infty}(\frac{\rho \tau_1'}{dt})^i$ converges to $\frac{\rho \tau_1'}{dt-\rho \tau_1'}$. \vspace{1mm}\\
Similar to $\textbf{d}_1$, the other component $d_3$ is as follows:
\begin{equation}
\begin{split}
d_3 & \leq \frac{\tau_3'}{dt- \tau_3'} [  \partial_t E^{I,n}_{c}-\bigtriangledown \cdot \tilde{\bigtriangledown} E^{I,n,\theta}_{c} + (\textbf{u}_h^{n} \cdot \bigtriangledown) E^{I,n,\theta}_{c} + (E^{I,n}_{\textbf{u}} \cdot \bigtriangledown) c^{n,\theta}+\alpha E^{I,n,\theta}_{c} \\
& \quad  +\partial_t E^{A,n}_{c}-\bigtriangledown \cdot \tilde{\bigtriangledown} E^{A,n,\theta}_{c} + (\textbf{u}_h^{n} \cdot \bigtriangledown) E^{A,n,\theta}_{c} + (E^{A,n}_{\textbf{u}} \cdot \bigtriangledown)c^{n,\theta} +\alpha E^{A,n,\theta}_{c} ]
\end{split}
\end{equation}
It is clearly seen in the expansion of $\textbf{d}_1$ and $d_3$ that the terms in $I_3$ and $I_4$ exactly match with the terms in  $I_2$ and $I_1$ respectively. Hence their estimations also follow the same way as we have done earlier. Therefore skipping the repetition of presenting same results, here we have mentioned the estimated results only for one term from each of $I_3$ and $I_4$ in the following. Denoting first term of $I_3$ by the notation $I_3^{1}$ we have the estimated result as follows:
\begin{equation}
\begin{split}
I_3^1 & = \sum_{k=1}^{n_{el}} (\frac{\rho \tau_{1k}}{dt+\rho \tau_{1k}} I_{d \times d} \{\rho \partial_t E^{I,n}_{\textbf{u}}+  \rho (E^{I,n}_{\textbf{u}} \bigtriangledown)\textbf{u}^{n,\theta}+ \rho (\textbf{u}_h^{n} \cdot \bigtriangledown)E^{I,n,\theta}_{\textbf{u}}  \\
& \quad -\mu(c^n) \Delta E^{I,n,\theta}_{\textbf{u}}+ \bigtriangledown E^{I,n,\theta}_p \}, E^{A,n,\theta}_{\textbf{u}})_{\Omega_k}\\
& \leq  h^2 \frac{\rho^2 C C_{\tau_1}}{(T_0- \rho  C_{\tau_{1}})}[ \sum_{i=1}^d \{\sum_{k=1}^{n_{el}} B_{1k}^i \}(\|u_i^{n+1}\|_2+\|u_i^n\|_2)] +  \frac{ \rho \mid \tau_{1}\mid}{(T_0- \rho C_{\tau_1})} \sum_{i=1}^d (\sum_{k=1}^{n_{el}} B_{1k}^i) \\
& \quad  [\{\rho h^2 \|\frac{\partial u_i^{n, \theta}}{\partial x_i}\| +\rho h \|u_i^{n}\| +\rho h B^i_{1k} + \mu_u + h^3 (\frac{1+\theta}{2} \|u_i^{n+1}\|_2 + \frac{1-\theta}{2} \|u_i^n\|_2)\}\\
& \quad   (\frac{1+\theta}{2} \|u_i^{n+1}\|_2 + \frac{1-\theta}{2} \|u_i^n\|_2) +  h (\frac{1+\theta}{2} \|p^{n+1}\|_1 + \frac{1-\theta}{2} \|p^n\|_1)  ]
\end{split}
\end{equation}
and denoting first term of $I_4$ by $I_4^1$ we have the estimated result in the following
\begin{equation}
\begin{split}
I_4^1 & =   \sum_{k=1}^{n_{el}} (\frac{\rho \tau_{1k}^2}{dt+\rho \tau_{1k}} I_{d \times d} \{\rho \partial_t E^{I,n}_{\textbf{u}}+  \rho (E^{I,n}_{\textbf{u}} \cdot \bigtriangledown)\textbf{u}^{n,\theta}+ \rho (\textbf{u}_h^{n} \cdot \bigtriangledown)E^{I,n,\theta}_{\textbf{u}} -\mu(c^n)  \\
& \quad \Delta E^{I,n,\theta}_{\textbf{u}}+ \bigtriangledown E^{I,n,\theta}_p \},  \rho(\textbf{u}_h \cdot \bigtriangledown) E^{A,n,\theta}_{\textbf{u}} + \mu(c) \Delta E^{A,n,\theta}_{\textbf{u}} + \bigtriangledown E^{A,n,\theta}_{p})_{\Omega_k}\\
& \leq  h^2 \frac{\rho^2 C C_{\tau_1}^2}{(T_0-\rho C_{\tau_{1}})} (\sum_{k=1}^{n_{el}} D_{B_{1k}}) \{\|\textbf{u}^{n+1}\|_2+\|\textbf{u}^n\|_2 \}  +  \frac{ \rho \mid \tau_{1}\mid^2}{(T_0- \rho C_{\tau_1})} (\sum_{k=1}^{n_{el}} D_{B_{1k}}) [ \sum_{i=1}^d\\
& \quad  \{\rho h^2 \|\frac{\partial u_i^{n, \theta}}{\partial x_i}\| +\rho h \|u_i^{n}\| +\rho h B^i_{1k} + \mu_u + h^3 (\frac{1+\theta}{2} \|u_i^{n+1}\|_2 + \frac{1-\theta}{2} \|u_i^n\|_2)\}\\
& \quad   (\frac{1+\theta}{2} \|u_i^{n+1}\|_2 + \frac{1-\theta}{2} \|u_i^n\|_2) +  h (\frac{1+\theta}{2} \|p^{n+1}\|_1 + \frac{1-\theta}{2} \|p^n\|_1)  ]
\end{split}
\end{equation}
Now the estimations of the remaining terms are quite obvious. Therefore we directly add those results while combining them into (43) at last. Finally the last term containing truncation error can be estimated as follows:
\begin{equation}
\begin{split}
(\textbf{TE}^{n,\theta}, E^{A,n,\theta}_{\textbf{U}}) & =(\textbf{TE}_1^{n,\theta}, E^{A,n,\theta}_{\textbf{u}})+ (TE_2^{n,\theta}, E^{A,n,\theta}_{c})\\
& \leq \frac{ \epsilon_9}{2}( \|\textbf{TE}^{n,\theta}_1\|^2+ \|TE^{n,\theta}_2\|^2) + \frac{1}{2 \epsilon_9} (\|E^{A,n,\theta}_{\textbf{u}}\|^2 +\|E^{A,n,\theta}_{c}\|^2)
\end{split}
\end{equation}
This completes estimation of all the terms in the $RHS$ of (43).  Now we start with putting all the bounds, obtained for each of the terms in the right hand side of (43). Then we take out few common terms in the left hand side and consequently we have left with those terms multiplied by $h^2, \mid \tau_1 \mid$ and $ \mid \tau_3 \mid$. Now we multiply both sides by 2 and taking integration over $(t^n,t^{n+1})$ for $n$=0,1,...,$(N-1)$ to both the sides. Finally we have (43) as follows:
\begin{multline}
\{1- \frac{2 T C_{\tau_3}(1-C_{\tau_3})}{T_0+\rho C_{\tau_3}}\sum_{k=1}^{n_{el}}D_{B_{3k}}- \frac{4 T C_{\tau_3}}{T_0-\rho C_{\tau_3}} \sum_{k=1}^{n_{el}} B_{2k}\} \sum_{n=0}^{N-1} (\|E^{A,n+1}_{c}\|^2 -\|E^{A,n}_{c}\|^2)\\
 + \rho \{ 1- \frac{2(1+\rho C_{\tau_1}) T C_{\tau_1}}{T_0-\rho C_{\tau_1}} \sum_{k=1}^{n_{el}} D_{B_{1k}} - \frac{2(1+\rho)T C_{\tau_1}}{T_0-\rho C_{\tau_1}} \sum_{i=1}^d  \sum_{k=1}^{n_{el}} B_{1k}^i \} \sum_{n=0}^{N-1} (\|E^{A,n+1}_{\textbf{u}}\|^2\\
  -\|E^{A,n}_{\textbf{u}}\|^2)  +  \{ 2\mu_l - \frac{1}{\epsilon_1}- \frac{\mu_u}{\epsilon_2}- \frac{4 C_2}{\epsilon_8}- 4C_2' -2 C_{\tau_2}\}  \sum_{n=0}^{N-1} \int_{t^n}^{t^{n+1}} \mid E^{A,n,\theta}_{\textbf{u}} \mid_1^2 dt \\
   +\{2 \sigma - \frac{\sigma}{\epsilon_4}- \frac{2 C_2}{\epsilon_7}- \frac{4 C_2}{\epsilon_8} -4 C_2'-\frac{1}{\epsilon_9} \} \sum_{n=0}^{N-1} \int_{t^n}^{t^{n+1}} \| E^{A,n,\theta}_{\textbf{u}}\|^2 dt   \\
+\{2D_l- \frac{D_m}{\epsilon_2}- 4 \bar{C}_1^2\} \sum_{n=0}^{N-1} \int_{t^n}^{t^{n+1}} \mid E^{A,n,\theta}_c \mid_1^2 dt +\\
\end{multline}
\begin{multline}
  \{ 2 \alpha - \frac{\alpha}{\epsilon_5} - \frac{4 \bar{C}_1^2}{\epsilon_6}-\frac{1}{\epsilon_9} \} \sum_{n=0}^{N-1} \int_{t^n}^{t^{n+1}} \| E^{A,n,\theta}_c\|^2 dt \\
 \leq  h^2 \sum_{n=0}^{N-1} \int_{t^n}^{t^{n+1}}[ 2 C^2 \epsilon_1 (\frac{1+\theta}{2} \|p^{n+1}\|_2 +\frac{1-\theta}{2} \|p^{n}\|_2)^2 + \{ 2C^2 \mu_u \epsilon_2 + h^2 \sigma \epsilon_4 +  2C_2 \epsilon_7 \\
+3 C_2 h^2 \epsilon_8 + C_{\tau_2} \} (\frac{1+\theta}{2} \|\textbf{u}^{n+1}\|_2 +\frac{1-\theta}{2} \|\textbf{u}^{n}\|_2)^2 + \{ C^2 D_m \epsilon_3+ h^2 \alpha \epsilon_5 + \bar{C}_1^2 C^2 \epsilon_6 \}  \\
(\frac{1+\theta}{2} \|c^{n+1}\|_2 +\frac{1-\theta}{2}  \|c^{n}\|_2)^2 + 2\frac{\rho C C_{\tau_1}}{T_0- \rho C_{\tau_1}} \{ (1+ \rho C_{\tau_1}) (\sum_{k=1}^{n_{el}} D_{B_{1k}}) + (1+ \rho)  \\
(\sum_{i=1}^d \sum_{k=1}^{n_{el}}B_{1k}^i) \}(\|\textbf{u}^{n+1}\|_2 + \|\textbf{u}^n\|_2)  +\frac{2 C C_{\tau_3}}{T_0-  C_{\tau_3}} \{\sum_{k=1}^{n_{el}} D_{B_{3k}}+(1+C_{\tau_3})\sum_{k=1}^{n_{el}} B_{2k} \}\\
(\|c^{n+1}\|_2+\|c^n\|_2)] dt  + 2 \mid \tau_1 \mid \sum_{n=0}^{N-1} \int_{t^n}^{t^{n+1}} [ \frac{T C + \rho C_{\tau_1}}{T_0- \rho C_{\tau_1}} \sum_{k=1}^{n_{el}} D_{B_{1k}} + \frac{1+\rho}{T_0- \rho C_{\tau_1}} \sum_{i=1}^d \\
 \sum_{k=1}^{n_{el}} B_{1k}^i] [\{\rho h^2 \|\frac{\partial u_i^{n, \theta}}{\partial x_i}\| +\rho h \|u_i^{n,\epsilon\theta}\| +\rho h B^i_{1k} + \mu_u + h^3 (\frac{1+\theta}{2} \|u_i^{n+1}\|_2 + \frac{1-\theta}{2} \|u_i^n\|_2)\}\\
   (\frac{1+\theta}{2} \|u_i^{n+1}\|_2 + \frac{1-\theta}{2} \|u_i^n\|_2) +  h (\frac{1+\theta}{2} \|p^{n+1}\|_1 + \frac{1-\theta}{2} \|p^n\|_1)  ] dt + \\
 \mid \tau_1 \mid  \sum_{n=0}^{N-1} \int_{t^n}^{t^{n+1}} \frac{ 2T}{T_0-\rho C_{\tau_1}} [(1+\rho) \sum_{k=1}^{n_{el}} D_{B_{1k}} \bar{D}_{B_{1k}}+ (1+\rho C_{\tau_1}) \sum_{i=1}^d  \sum_{k=1}^{n_{el}}  \bar{D}_{B_{1k}} B_{1k}^i] dt \\
+2 \mid \tau_3 \mid \sum_{n=0}^{N-1} \int_{t^n}^{t^{n+1}} [ \frac{2C}{T_0- C_{\tau_3}} \sum_{k=1}^{n_{el}} D_{B_{3k}} + \frac{1+ C_{\tau_3}}{T_0- C_{\tau_3}} \sum_{k=1}^{n_{el}} B_{2k}] [ \sum_{i=1}^d \{h^2  \|\frac{\partial c^{n,\theta}}{\partial x_i}\|\\
  (\frac{1+\theta}{2}\|u_i^{n+1}\|_2+\frac{1-\theta}{2}\|u_i^n\|_2) + (D_{im}+h \bar{D}_{im}+ \mid \alpha\mid  h^2+ h\|u_i^{n,\theta}\| + hB_{1k}^i\\
+ h^3  (\frac{1+\theta}{2} \|u_i^{n+1}\|_2+\frac{1-\theta}{2}\|u_i^n\|_2))(\frac{1+\theta}{2}\|c^{n+1}\|_2+\frac{1-\theta}{2} \|c^n\|_2)\}] dt \\
+ \mid \tau_3 \mid  \sum_{n=0}^{N-1} \int_{t^n}^{t^{n+1}} \frac{2C}{T_0 - C_{\tau_3}} [\sum_{k=1}^{n_{el}} \bar{D}_{B_{3k}} D_{B_{3k}} + (1+ C_{\tau_3}) \sum_{k=1}^{n_{el}} \bar{D}_{B_{3k}} B_{2k}] dt \\
+  \epsilon_9 \sum_{n=0}^{N-1} \int_{t^n}^{t^{n+1}}( \|\textbf{TE}^{n,\theta}_1\|^2+ \|TE^{n,\theta}_2\|^2) dt \hspace{42mm}
\end{multline}
We can choose the values of the arbitrary parameters in such a manner that we can make all the coefficients in the left hand side positive. In order to satisfy such condition it is inevitable to choose the characteristic lengths small. Now after taking minimum of all the coefficients in left hand side, let us divide both the sides with that minimum, which turns out to be a positive real number. Applying assumption $\textbf{(iv)}$ it can be seen that $\|\textbf{u}^n\|_2$, $\|p^n\|_1$ and $\|c^n\|_2$ are bounded for $n=0,1,2,...,N$.  The initial conditions considered in section 3.1, imply $\|E^{A,0}_{\textbf{u}}\|=0$ and $\|E^{A,0}_c\|=0$. \vspace{1mm}\\
After performing all these intermediate steps and applying the properties (15)-(16) on truncation errors  we finally arrive at the following expression since $\tau_1$ and $\tau_3$ are of order $h^2$:
\begin{multline}
\|E^{A,N}_{\textbf{u}}\|^2 +\|E^{A,N}_{c}\|^2 +\sum_{n=0}^{N-1}  \int_{t^n}^{t^{n+1}}   \|E^{A,n,\theta}_{\textbf{u}} \|_1^2 dt+ \sum_{n=0}^{N-1}  \int_{t^n}^{t^{n+1}}  \|E^{A,n,\theta}_{c} \|^2_1 dt\\
 \leq C(T,\textbf{u},p,c) (h^2+dt^{2r}) \hspace{30mm}
\end{multline}
This implies
\begin{equation}
\boxed{\|E^{A}_{\textbf{u}}\|_{\tilde{\textbf{V}}}^2  + \|E^{A}_{c}\|_{\tilde{\textbf{V}}}^2 \leq C(T,\textbf{u},p,c) (h^2+dt^{2r})}
\end{equation}
where
\begin{equation}
    r=
    \begin{cases}
      1, & \text{if}\ \theta=1 \\
      2, & \text{if}\ \theta=0
    \end{cases}
  \end{equation}
We have used the fact that $\sum_{n=0}^{N-1} \int_{t^n}^{t^{n+1}} M dt \leq M T  $. This completes the first part of the proof. \vspace{2mm}\\
\textbf{Second part} Using this above result we are going to estimate auxiliary error part of pressure. We will use inf-sup condition to find estimate for $E_p^A$. Applying Galerkin orthogonality only for variational form of Navier-Stokes flow problem we have obtained \\
\begin{multline}
(\frac{\partial (\textbf{u}-\textbf{u}_h)}{\partial t}, \textbf{v}_{h})+ c(\textbf{u},\textbf{u},\textbf{v}_h)- c(\textbf{u}_h,\textbf{u}_h,\textbf{v}_h)+ 
 a_{NS} (\textbf{u}-\textbf{u}_h, \textbf{v}_h)- b(\textbf{v}_h, p-p_h)  =0 \\ 
 \end{multline}
Splitting of the errors implies the following 
\begin{multline}
b(\textbf{v}_h, p-I_hp)+ b(\textbf{v}_h, I_hp-p_h)  = (\partial_t E^A_{\textbf{u}}, \textbf{v}_{h})+ c(E^I_{\textbf{u}},\textbf{u},\textbf{v}_h)+c(E^A_{\textbf{u}},\textbf{u},\textbf{v}_h)+\\
 c(\textbf{u}_h,E^I_{\textbf{u}},\textbf{v}_h)+ c(\textbf{u}_h,E^A_{\textbf{u}},\textbf{v}_h)+ a_{NS}(E^I_\textbf{u}, \textbf{v}_h)+  a_{NS}(E^A_\textbf{u}, \textbf{v}_h)
\end{multline}
Without loss of generality considering the inclusion $\bigtriangledown \cdot V^h \subset Q^h$ and the property of the $L^2$ orthogonal projection of $I^h_p$ we have 
\begin{equation}
b(\textbf{v}_h, p-I_hp)= \int_{\Omega}(p-I_hp)(\bigtriangledown \cdot \textbf{v}_h)=0
\end{equation}
Now according to inf-sup condition we will have the following expression
\begin{equation}
\begin{split}
\|I_hp-p_h\|_{L^2(L^2)}^2& = \|E_p^A\|_{L^2(L^2)}^2 \\
& = \sum_{n=0}^{N-1} \int_{t^n}^{t^{n+1}} \|E_p^{A,n,\theta}\|^2 dt \\
& \leq  \sum_{n=0}^{N-1} \int_{t^n}^{t^{n+1}} \underset{\textbf{v}_h}{sup} \frac{b(\textbf{v}_h, E_p^{A,n,\theta})}{\|\textbf{v}_h\|_1} dt
\end{split}
\end{equation}
Using (74) on (73) we will have
\begin{equation}
\begin{split}
\sum_{n=0}^{N-1} \int_{t^n}^{t^{n+1}} b(\textbf{v}_h, E_p^{A,n,\theta}) dt & = \sum_{n=0}^{N-1} \int_{t^n}^{t^{n+1}} \{ (\frac{E^{A,n+1}_{\textbf{u}}-E^{A,n}_{\textbf{u}}}{dt},\textbf{v}_h) + a_{NS}(E^{I,n,\theta}_{\textbf{u}},\textbf{v}_h)\\
& \quad + c(E^{I,n}_{\textbf{u}},\textbf{u}^{n,\theta},\textbf{v}_h)+ c(E^{A,n}_{\textbf{u}},\textbf{u}^{n,\theta},\textbf{v}_h)+ \\
& \quad c(\textbf{u}_h^{n}, E^{I,n,\theta}_\textbf{u},\textbf{v}_h)+ c(\textbf{u}_h^{n}, E^{A,n,\theta}_\textbf{u},\textbf{v}_h)+ \\
& \quad  a_{NS}(E^{A,n,\theta}_{\textbf{u}},\textbf{v}_h)\}dt
\end{split}
\end{equation}
Now applying the results obtained in the previous part we will have,
\begin{equation}
\sum_{n=0}^{N-1} \int_{t^n}^{t^{n+1}} b(\textbf{v}_h, E_p^{A,n,\theta}) dt \leq C(T,\textbf{u},p,c)(h^2+ dt^{2r}) \|\textbf{v}_h\|_1
\end{equation}
Using this above result into (75), we will have the estimate for the pressure term
\begin{equation}
\|I_hp-p_h\|_{L^2(L^2)}^2 \leq C(T,\textbf{u},p,c) (h^2+dt^{2r})
\end{equation}
Now combining the results obtained in the first and second part we have finally arrived at the following auxiliary error estimate:
\begin{equation}
\|E^A_{\textbf{u}}\|^2_{\tilde{\textbf{V}}} + \|E^A_p\|_{L^2(L^2)}^2  + \|E^A_c\|^2_{\tilde{\textbf{V}}} \leq C(T,\textbf{u},p,c) (h^2+ dt^{2r})
\end{equation}
where
\begin{equation}
    r=
    \begin{cases}
      1, & \text{if}\ \theta=1 \\
      2, & \text{if}\ \theta=0
    \end{cases}
  \end{equation}
This completes the proof.
\end{proof} 
 
\begin{theorem}\textbf{(Apriori error estimate)}
Assuming the same condition as in the previous theorem, 
\begin{equation}
\|\textbf{u}-\textbf{u}_h\|_{\tilde{\textbf{V}}}^2+\|p-p_h\|_{L^2(L^2)}^2 + \|c-c_h\|^2_{\tilde{\textbf{V}}} \leq C' (h^2+ dt^{2r})
\end{equation}
where $C'$ depends on T, $\textbf{u}$,p,c and
\begin{equation}
    r=
    \begin{cases}
      1, & \text{if}\ \theta=1 \\
      2, & \text{if}\ \theta=0
    \end{cases}
  \end{equation}
\end{theorem}
\begin{proof}
By applying triangle inequality, the interpolation inequalities and the result of the previous theorem we will have,
\begin{multline}
\|\textbf{u}-\textbf{u}_h\|_{\tilde{\textbf{V}}}^2+\|p-p_h\|_{L^2(L^2)}^2 +\|c-c_h\|^2_{\tilde{\textbf{V}}} \\
 = \|E^I_{\textbf{u}}+E^A_{\textbf{u}}\|_{\tilde{\textbf{V}}}^2 +  \|E^I_{p}+E^A_{p}\|_{L^2(L^2)}^2+ \|E^I_{c}+E^A_{c}\|_{\tilde{\textbf{V}}}^2  \hspace{31mm} \\
 \leq \bar{C} (\|E^I_{\textbf{u}}\|_{\tilde{\textbf{V}}}^2 +\|E^I_{p}\|_{L^2(L^2)}^2 +  \|E^I_{c}\|_{\tilde{\textbf{V}}}^2+ 
 \|E^A_{\textbf{u}}\|_{\tilde{\textbf{V}}}^2+ \|E^A_{p}\|_{L^2(L^2)}^2+\|E^A_{c}\|_{\tilde{\textbf{V}}}^2) \\
 \leq C'(T,\textbf{u},p,c)(h^2+ dt^{2r}) \hspace{65mm}
\end{multline}
This completes $apriori$ error estimation.
\end{proof} 

\subsection{Aposteriori error estimation}
In this section we are going to derive residual based aposteriori error estimation. This estimation is also comprised of two parts similar to the auxiliary apriori error estimate derived in the earlier section. \vspace{1mm} \\
\begin{theorem} 
For computed velocity $\textbf{u}_h$, pressure $p_h$ and concentration $c_h$ belonging to $(V^h)^d \times Q^h \times V^h$ satisfying (15)-(16), assume $dt$ is sufficiently small and positive, and sufficient regularity of exact solution in equations (1)-(2). Then there exists a constant $\bar{C}$, independent of $\textbf{u},p,c$ and depending on the residual such that
\begin{equation}
\|\textbf{u}-\textbf{u}_h\|^2_{\tilde{\textbf{V}}} + \|p-p_h\|_{L^2(L^2)}^2  + \|c-c_h\|^2_{\tilde{\textbf{V}}} \leq \bar{C}(\textbf{R}) (h^2+ dt^{2r})
\end{equation}
where $\textbf{R}$ is the residual vector and
\begin{equation}
    r=
    \begin{cases}
      1, & \text{if}\ \theta=1 \\
      2, & \text{if}\ \theta=0
    \end{cases}
  \end{equation}
\end{theorem}
\begin{proof}
We estimate $aposteriori$ error by dividing the procedure into two parts. In the first part we find error bound corresponding to $velocity$ and $concentration$ followed by the second part estimating error associated with the $pressure$ term. Let us first introduce the residual vector corresponding to each equations 
\[
\textbf{R}=
  \begin{bmatrix}
 \textbf{f}-\{ \rho \frac{\partial \textbf{u}_h}{\partial t} + \rho (\textbf{u}_h \cdot \bigtriangledown) \textbf{u}_h - \mu(c) \Delta \textbf{u}_h  + \bigtriangledown p_h \} \\
    -\bigtriangledown \cdot \textbf{u}_h \\
 g-(\frac{\partial c_h}{\partial t} - \bigtriangledown \cdot \tilde{\bigtriangledown} c_h + \textbf{u} \cdot \bigtriangledown c_h + \alpha c_h )
  \end{bmatrix}
   = 
  \begin{bmatrix}
 \textbf{R}_1\\
  R_2 \\
  R_3
  \end{bmatrix}
\]
\textbf{First part}: We have $\forall \textbf{V} \in \textbf{V}_F$
\begin{equation}
\mu_l \mid \textbf{v} \mid_1^2 + D_{\alpha} \|d\|_1^2 \leq B(\textbf{u};\textbf{V},\textbf{V})= a_{NS}(\textbf{v},\textbf{v})+a_T(d,d)
\end{equation} 
Since $\textbf{e} \in \textbf{V}_F$ we substitute the errors $e_{\textbf{u}},e_c$ into the above relation:
\begin{equation}
\mu_l \|e_{\textbf{u}}\|_1^2+ D_{\alpha} \|e_c\|_1^2  \leq a_{NS}(e_\textbf{u},e_\textbf{u})+a_T(e_c,e_c)+ \mu_l \|e_{\textbf{u}}\|^2\\
\end{equation} 
By adding few terms in both sides the above equation becomes
\begin{multline}
 (\frac{\partial e_{\textbf{u}}}{\partial t},e_{\textbf{u}})+(\frac{\partial e_{c}}{\partial t},e_{c})+\mu_l \|e_{\textbf{u}}\|_1^2+  D_{\alpha} \|e_c\|_1^2 \\
 \leq   (\frac{\partial e_{\textbf{u}}}{\partial t},e_{\textbf{u}})+(\frac{\partial e_{c}}{\partial t},e_{c})+c(\textbf{u}, e_\textbf{u},e_\textbf{u})+a_{NS}(e_\textbf{u},e_\textbf{u})+a_T(e_c,e_c) \hspace{10mm} \\
 +b(e_\textbf{u},e_p)-b(e_\textbf{u},e_p)+ \mu_l \|e_{\textbf{u}}\|^2
\end{multline} 
Now first we will find a lower bound of $LHS$ and then upper bound for $RHS$ and finally combining them we will get $aposteriori$ error estimate. To find the lower bound the $LHS$ can be written as
\begin{multline}
LHS = (\frac{e^{n+1}_{\textbf{u}}-e^{n}_{\textbf{u}}}{dt},e_{\textbf{u}}^{n,\theta})+ (\frac{e_{c}^{n+1}-e_{c}^n}{dt},e_{c}^{n,\theta})+ 
\mu_l \|e_{\textbf{u}}^{n,\theta}\|_1^2 +    D_{\alpha} \|e_{c}^{n,\theta}\|^2_1
\end{multline}
Applying (32) on first two terms of $LHS$ we have the following relations
\begin{equation}
\begin{split}
(\frac{e_{\textbf{u}}^{n+1}-e_{\textbf{u}}^n}{dt},e_{\textbf{u}}^{n,\theta}) & \geq \frac{1}{2 dt}(\|e_{\textbf{u}}^{n+1}\|^2-\|e_{\textbf{u}}^n\|^2) \\
\end{split}
\end{equation}
and
\begin{equation}
\begin{split}
(\frac{e_{c}^{n+1}-e_{c}^n}{dt},e_{c}^{n,\theta}) & \geq \frac{1}{2 dt}(\|e_{c}^{n+1}\|^2-\|e_{c}^n\|^2) \\
\end{split}
\end{equation}
Hence
\begin{multline}
 \frac{1}{2 dt}(\|e_{\textbf{u}}^{n+1}\|^2-\|e_{\textbf{u}}^n\|^2)+\frac{1}{2 dt}(\|e_{c}^{n+1}\|^2-\|e_{c}^n\|^2)+\mu_l \|e_{\textbf{u}}^{n,\theta}\|_1^2 +   D_{\alpha}\|e_{c}^{n,\theta}\|^2_1\\
  \leq LHS \leq RHS
\end{multline}
Now our job is to find upper bound for $RHS$ and to reach at the desired estimates let us divide it into two broad parts by splitting errors in each of the terms in the following way:
\begin{equation}
\begin{split}
RHS &= [(\frac{e_{\textbf{u}}^{n+1}-e_{\textbf{u}}^n}{dt},E_{\textbf{u}}^{I,n,\theta}) +(\frac{e_c^{n+1}-e_c^n}{dt},E^{I,n,\theta}_c)+c(\textbf{u}^{n},e_\textbf{u}^{n,\theta}, E^{I,n,\theta}_\textbf{u}) +\\
& \quad a_{NS}(e^{n,\theta}_\textbf{u},E^{I,n,\theta}_\textbf{u})+ 
b(e^{n,\theta}_\textbf{u},E^{I,n,\theta}_{p})-  b(E^{I,n,\theta}_\textbf{u},e^{n,\theta}_p) +a_T(e^{n,\theta}_c,E^{I,n,\theta}_{c}) ] + 
\\
& \quad [(\frac{e_{\textbf{u}}^{n+1}-e_{\textbf{u}}^n}{dt},E_{\textbf{u}}^{A,n,\theta}) + (\frac{e^{n+1}_{c}-e_c^n}{ dt},E^{A,n,\theta}_{c})+c(\textbf{u}^{n},e_\textbf{u}^{n,\theta}, E^{A,n,\theta}_\textbf{u}) +
 \\
& \quad a_{NS}(e^{n,\theta}_\textbf{u},E^{A,n,\theta}_\textbf{u})+ b(e^{n,\theta}_\textbf{u} ,E^{A,n,\theta}_{p})-  b(E^{A,n,\theta}_\textbf{u},e^{n,\theta}_p) +a_T(e^{n,\theta}_c,E^{A,n,\theta}_{c})] \\
& \quad + \mu_l \|e_{\textbf{u}}^{n,\theta}\|^2 \\
&= RHS^I + RHS^A  + \mu_l \|e_{\textbf{u}}^{n,\theta}\|^2\\
\end{split}
\end{equation}
Our aim is to bring residual into context and for this purpose $RHS^I$ involving interpolation error terms can be written as follows:
\begin{equation}
\begin{split}
RHS^I & = [\rho (\frac{e_{\textbf{u}}^{n+1}-e_{\textbf{u}}^n}{dt},E^{I,n,\theta}_\textbf{u}) + (\frac{e_c^{n+1}-e_c^n}{dt}, E^{I,n,\theta}_c) + c(\textbf{u}^{n}, \textbf{u}^{n,\theta},E^{I,n,\theta}_\textbf{u})-\\
& \quad c(\textbf{u}^{n}_h, \textbf{u}^{n,\theta}_h,E^{I,n,\theta}_\textbf{u}) + a_{NS}(e^{n,\theta}_{\textbf{u}}, E^{I,n,\theta}_\textbf{u}) - b(E^{I,n,\theta}_{\textbf{u}},e^{n,\theta}_{p})+ b(e^{n,\theta}_{\textbf{u}},E^{I,n,\theta}_p) \\
& \quad + a_{T}(e_c^{n,\theta}, E^{I,n,\theta}_c) ]-c(e_\textbf{u}^{n}, \textbf{u}_h^{n,\theta}, E^{I,n,\theta}_\textbf{u})\\
& =RHS^I_1 - c(e_\textbf{u}^{n}, \textbf{u}_h^{n,\theta}, E^{I,n,\theta}_\textbf{u}) 
\end{split}
\end{equation}
The bracketed term in the above equation is denoted by $RHS^I_1$.  $RHS^A$ involving auxiliary part of error can also be decomposed in the similar manner as above and let us denote the alike term corresponding to $RHS^A$ by $RHS^A_1$. Therefore
\begin{equation}
RHS^A =RHS^A_1 - c(e_\textbf{u}^{n}, \textbf{u}_h^{n,\theta}, E^{A,n,\theta}_\textbf{u}) \\
\end{equation}
Hence combining these above two results (93) becomes
\begin{equation}
\begin{split}
RHS & = RHS^I_1 + RHS^A_1 - c(e_\textbf{u}^{n}, \textbf{u}_h^{n,\theta}, e_\textbf{u}^{n,\theta})+ \mu_l \|e_{\textbf{u}}^{n,\theta}\|^2\\
& =   RHS^I_1 + RHS^A_1 - c(e_\textbf{u}^{n}, \textbf{u}^{n,\theta}, e_\textbf{u}^{n,\theta}) + c(e_\textbf{u}^{n}, e_\textbf{u}^{n,\theta}, e_\textbf{u}^{n,\theta})+ \mu_l \|e_{\textbf{u}}^{n,\theta}\|^2\\
& =   RHS^I_1 + RHS^A_1 - c(e_\textbf{u}^{n}, \textbf{u}^{n,\theta}, e_\textbf{u}^{n,\theta})+ \mu_l \|e_{\textbf{u}}^{n,\theta}\|^2
\end{split}
\end{equation}
Property \textbf{(a)} of trilinear term $c(\cdot, \cdot, \cdot)$ implies $c(e_\textbf{u}^{n,\epsilon \theta}, e_\textbf{u}^{n,\theta}, e_\textbf{u}^{n,\theta})=0$. \\
In the most general way for $Navier$-$Stokes$ flow problem we have for all $\textbf{v} \in (V)^d$
\begin{multline}
\rho (\frac{e_{\textbf{u}}^{n+1}- e_{\textbf{u}}^{n}}{dt}, \textbf{v})+ c(\textbf{u}^{n}, \textbf{u}^{n,\theta}, \textbf{v})- c(\textbf{u}^{n}_h, \textbf{u}^{n,\theta}_h, \textbf{v})+ a_{NS}(e^{n,\theta}_{\textbf{u}},\textbf{v})-b(\textbf{v}, e_p^{n,\theta})\\
= \int_{\Omega} \textbf{R}_1^{n,\theta} \cdot \textbf{v} \hspace{90mm}\\
Similarly \hspace{2mm} \int_{\Omega} (\bigtriangledown \cdot e_\textbf{u}^{n,\theta})q = \int_{\Omega} R_2^{n,\theta} q \hspace{2mm} \forall q \in Q \hspace{47mm}\\
\int_{\Omega} (\frac{ e_c^{n+1}-e_c^{n}}{dt} d + \tilde{\bigtriangledown} e_c^{n,\theta} \cdot \bigtriangledown d + d \textbf{u}^{n} \cdot \bigtriangledown e_c^{n,\theta} + \alpha e_c^{n,\theta} d)= \int_{\Omega} R_3^{n,\theta} d \hspace{3 mm} \forall d \in V
\end{multline}
Now substituting $\textbf{v} ,q,d$ in the above expressions by $E^{I,n,\theta}_{\textbf{u}}, E^{I,n,\theta}_{p}, E^{I,n,\theta}_{c}$ respectively, we have the $RHS^I_1$ as,
\begin{equation}
\begin{split}
RHS^I_1 & = \int_{\Omega} (\textbf{R}_1^{n,\theta} \cdot E^{I,n,\theta}_{\textbf{u}}+ R_2^{n,\theta} E^{I,n,\theta}_{p}+ R_3^{n,\theta} E^{I,n,\theta}_{c} )\\
& \leq h^2 \{ \|\textbf{R}_1^{n,\theta}\| (\frac{1+\theta}{2} \| \textbf{u}^{n+1} \|_2 + \frac{1-\theta}{2} \| \textbf{u}^n \|_2) + C_2 \|R_2^{n,\theta}\| (\frac{1+\theta}{2} \|p^{n+1}\|_1  \\
& \quad  + \frac{1-\theta}{2} \|p^n\|_1) + \|R_3^{n,\theta}\|  (\frac{1+\theta}{2}\|c_{n+1}\|_2 +\frac{1-\theta}{2} \|c^n\|_2)\}\\
& \leq h^2( \bar{C}_1 \|\textbf{R}_1^{n,\theta}\| + \bar{C}_2 \|R_2^{n,\theta}\|+ \bar{C}_3 \|R_3^{n,\theta}\|) \\
\end{split}
\end{equation}
The parameters $\bar{C}_i$, for i=1,2,3,4, are coming from imposing assumption  \textbf{(iv)}. Now we are going to estimate of $RHS^A_1$. For that we employ $subgrid$ formulation (8). Subtracting (8) from the variational finite element formulation satisfied by the exact solution we have $\forall \textbf{V}_h \in \textbf{V}_F^h $
\begin{multline}
\rho (\frac{e_{\textbf{u}}^{n+1}-e_{\textbf{u}}^n}{dt}, \textbf{v}_h)+ (\frac{e_c^{n+1}-e_c^{n}}{dt},d_h) + c(\textbf{u}^{n},\textbf{u}^{n,\theta},\textbf{v}_h)- c(\textbf{u}^{n}_h,\textbf{u}^{n,\theta}_h,\textbf{v}_h) \\
+a_{NS}(e_{\textbf{u}}^{n,\theta},\textbf{v}_h)- b(\textbf{v}_h, e^{n,\theta}_p)+ b(e_{\textbf{u}}^{n,\theta}, q_h) + a_{T}(e_c^{n,\theta},d_h)\\
 = \sum_{k=1}^{n_{el}} \{(\tau_k'(\textbf{R}^{n,\theta}+\textbf{d}), -\mathcal{L}^* \textbf{V}_h)_{\Omega_k} - ((I-\tau_k^{-1}\tau_k) \textbf{R}^{n,\theta}, \textbf{V}_h)_{\Omega_k} + (\tau_k^{-1}\tau_k \textbf{d}, \textbf{V}_h)_{\Omega_k} \} \\
  +(\textbf{TE}^{n,\theta}_1, \textbf{v}_{h})+(TE^{n,\theta}_2,d_h)\\
\end{multline}
\begin{multline}
= \sum_{k=1}^{n_{el}} [( \tau_1' I_{d \times d} \{\textbf{R}_1^{n,\theta}+\textbf{d}_1 \}, \rho (\textbf{u}_h \cdot \bigtriangledown) \textbf{v}_{h}+\mu(c) \Delta \textbf{v}_{h}+ \bigtriangledown q_h)_{\Omega_k} + \tau_2'(R_2^{n,\theta}, \bigtriangledown \cdot \textbf{v}_h)_{\Omega_k} \\
\quad   + \tau_3'(R_3^{n,\theta}+d_3, \bigtriangledown \cdot \tilde{\bigtriangledown} d_h + \textbf{u}_h \cdot \bigtriangledown d_h - \alpha d_h)_{\Omega_k}+ ((1-\tau_1^{-1}\tau_1') I_{d \times d} \textbf{R}_1^{n,\theta},\textbf{v}_{h})_{\Omega_k}\\
 +( (1-\tau_3^{-1}\tau_3') R_3^{n,\theta},d_{h})_{\Omega_k} +(\tau_1^{-1}\tau_1' I_{d \times d} \textbf{d}_1,\textbf{v}_{h})_{\Omega_k}+\tau_3^{-1}\tau_3' (d_3,d_{h})_{\Omega_k}]  \\
+ (\textbf{TE}_1^{n,\theta},\textbf{v}_{h}) +(TE_2^{n,\theta},d_h) \hspace{55mm}
\end{multline} 
Now substituting $\textbf{V}_h$ by $(E^{A,n,\theta}_{\textbf{u}},  E^{A,n,\theta}_{p}, E^{A,n,\theta}_{c})$ in the above equation we have $RHS^A_1$ as follows 
\begin{equation}
\begin{split}
RHS^A_1 & = \sum_{k=1}^{n_{el}}[( \tau_1' I_{d \times d} \{\textbf{R}_1^{n,\theta}+\textbf{d}_1 \}, \rho (\textbf{u}_h \cdot \bigtriangledown) E^{A,n,\theta}_{\textbf{u}}+\mu(c) \Delta E^{A,n,\theta}_{\textbf{u}}+ \bigtriangledown E^{A,n,\theta}_{p})_{\Omega_k}  \\
& \quad + \tau_2'(R_2^{n,\theta}, \bigtriangledown \cdot E^{A,n,\theta}_{\textbf{u}})_{\Omega_k} + \tau_3'(R_3^{n,\theta}+d_4, \bigtriangledown \cdot \tilde{\bigtriangledown} E^{A,n,\theta}_{c} + \textbf{u}_h \cdot \bigtriangledown E^{A,n,\theta}_{c} \\
& \quad - \alpha E^{A,n,\theta}_{c})_{\Omega_k} + ((1-\tau_1^{-1}\tau_1') I_{d \times d} \textbf{R}_1^{n,\theta}, E^{A,n,\theta}_{\textbf{u}})_{\Omega_k} + (1-\tau_3^{-1}\tau_3') (R_3^{n,\theta}, \\
& \quad E^{A,n,\theta}_{c})_{\Omega_k}+(\tau_1^{-1}\tau_1'I_{d \times d}\textbf{d}_1,E^{A,n,\theta}_{\textbf{u}})_{\Omega_k}+\tau_3^{-1}\tau_3' (d_3,E^{A,n,\theta}_{c})_{\Omega_k}] + \\
& \quad (\textbf{TE}_1^{n,\theta},E^{A,n,\theta}_{\textbf{u}})  +(TE_2^{n,\theta},E^{A,n,\theta}_{c})
\end{split}
\end{equation}
Now we estimate each term separately. We use the results mentioned earlier during $apriori$ error estimation.
\begin{equation}
\begin{split}
 \sum_{k=1}^{n_{el}}( \tau_1' I_{d \times d} \textbf{R}_1^{n,\theta}, \rho (\textbf{u}_h \cdot \bigtriangledown) E^{A,n,\theta}_{\textbf{u}}+\mu(c) \Delta E^{A,n,\theta}_{\textbf{u}}+ \bigtriangledown E^{A,n,\theta}_{p})_{\Omega_k} \\
  \leq \frac{\mid \tau_1 \mid T}{T_0- \rho C_{\tau_1}} (\sum_{k=1}^{n_{el}} D_{B_{1k}})\|\textbf{R}_1^{n,\theta}\| \hspace{49mm} \\
\sum_{k=1}^{n_{el}} \tau_3'(R_3^{n,\theta},\bigtriangledown \cdot \tilde{\bigtriangledown} E^{A,n,\theta}_{c} + \textbf{u}_h \cdot \bigtriangledown E^{A,n,\theta}_{c} - \alpha E^{A,n,\theta}_{c} )_k \hspace{14mm} \\
 \leq \frac{\mid \tau_3 \mid T}{T_0-  C_{\tau_3}} (\sum_{k=1}^{n_{el}} D_{B_{3k}}) \|R_3^{n,\theta}\| \hspace{50mm}\\ 
\end{split}
\end{equation}
and the other set of terms can be estimated as follows:
\begin{equation}
\begin{split}
\sum_{k=1}^{n_{el}}((1-\tau_1^{-1} \tau_1')I_{d \times d} \textbf{R}_1^{n,\theta}, E^{A,n,\theta}_{\textbf{u}})_{\Omega_k} & \leq \frac{\rho \mid \tau_1 \mid}{T_0-\rho C_{\tau_1}} (\sum_{i=1}^d \sum_{k=1}^{n_{el}}B_{1k}^i) \|\textbf{R}_1^{n,\theta}\| \\
\sum_{k=1}^{n_{el}}(1-\tau_3^{-1} \tau_3')(R_3^{n,\theta}, E^{A,n,\theta}_{c})_{\Omega_k} & \leq \frac{ \mid \tau_3 \mid}{T_0- C_{\tau_3}} (\sum_{k=1}^{n_{el}}B_{2k}) \|R_3^{n,\theta}\| \\
\end{split}
\end{equation} 
let us look into the form of the column vector $\textbf{d}$ which has components $\textbf{d}_1,d_2,d_3$. \vspace{1mm} \\
$\textbf{d}$= $\sum_{i=1}^{n+1}(\frac{1}{dt}M\tau_k')^i(\textbf{F} -M\partial_t \textbf{U}_h - \mathcal{L}\textbf{U}_h)=\sum_{i=1}^{n+1}(\frac{1}{dt}M\tau_k')^i \textbf{R}$ \vspace{2mm}\\
Hence we have the components $\textbf{d}_1= (\sum_{i=1}^{n+1}(\frac{1}{dt} \tau_1')^i) I_{d \times d} \textbf{R}_1^{n,\theta}$, $d_2=0$ and $d_3=(\sum_{i=1}^{n+1}(\frac{1}{dt} \tau_3')^i) R_3^{n,\theta}$ \vspace{2mm} \\ 
Now the terms containing the components of $\textbf{d}$ can be estimated in the following way:
\begin{multline}
\sum_{k=1}^{n_{el}} (\tau_1' I_{d \times d} \textbf{d}_1,\rho (\textbf{u}_h \cdot \bigtriangledown) E^{A,n,\theta}_{\textbf{u}}+\mu(c) \Delta E^{A,n,\theta}_{\textbf{u}}+ \bigtriangledown E^{A,n,\theta}_{p})_{\Omega_k} \\
 = \sum_{k=1}^{n_{el}}( \tau_1' \{\sum_{i=1}^{n+1}(\frac{1}{dt} \tau_1')^i\} I_{d \times d}\textbf{R}_1^{n,\theta},  \rho (\textbf{u}_h \cdot \bigtriangledown) E^{A,n,\theta}_{\textbf{u}}+\mu(c) \Delta E^{A,n,\theta}_{\textbf{u}}+ \bigtriangledown E^{A,n,\theta}_{p})_{\Omega_k}\\
 \leq \sum_{k=1}^{n_{el}}(\tau_1' \{\sum_{i=1}^{\infty}(\frac{1}{dt} \tau_1')^i\} I_{d \times d} \textbf{R}_1^{n,\theta}, \rho (\textbf{u}_h \cdot \bigtriangledown) E^{A,n,\theta}_{\textbf{u}}+\mu(c) \Delta E^{A,n,\theta}_{\textbf{u}}+ \bigtriangledown E^{A,n,\theta}_{p})_{\Omega_k}\\
 = \sum_{k=1}^{n_{el}}( \frac{\rho \tau_1^2}{dt+\rho \tau_1} I_{d \times d} \textbf{R}_1^{n,\theta},  \rho (\textbf{u}_h \cdot \bigtriangledown) E^{A,n,\theta}_{\textbf{u}}+\mu(c) \Delta E^{A,n,\theta}_{\textbf{u}}+ \bigtriangledown E^{A,n,\theta}_{p})_{\Omega_k} \hspace{8mm} \\
 \leq \frac{\mid \tau_1 \mid}{T_0-\rho C_{\tau_1}} \rho C_{\tau_1} (\sum_{k=1}^{n_{el}} D_{B_{1k}}) \|\textbf{R}_1^{n,\theta}\| \hspace{53mm}
\end{multline} 
Similarly the next few terms will follow the same way as above.
\begin{multline}
 \sum_{k=1}^{n_{el}} \tau_3'(d_3,\bigtriangledown \cdot \tilde{\bigtriangledown} E^{A,n,\theta}_{c} + \textbf{u}_h \cdot \bigtriangledown E^{A,n,\theta}_{c} - \alpha E^{A,n,\theta}_{c} )_k \hspace{35mm}\\
 \leq \frac{\mid \tau_3 \mid}{T_0- C_{\tau_3}}  C_{\tau_3} (\sum_{k=1}^{n_{el}} D_{B_{3k}}) \|R_3^{n,\theta}\| \hspace{63mm} \\
 \sum_{k=1}^{n_{el}} (\tau_1^{-1}\tau_1' I_{d \times d} \textbf{d}_1,E^{A,n,\theta}_{\textbf{u}})_{\Omega_k} \leq \frac{\rho \mid \tau_1 \mid}{T_0- \rho C_{\tau_1}} (\sum_{i=1}^d \sum_{k=1}^{n_{el}} B_{1k}^i) \|\textbf{R}_1^{n,\theta}\| \hspace{20mm} \\
 \sum_{k=1}^{n_{el}} \tau_3^{-1}\tau_3'(d_3,E^{A,n,\theta}_{c})_{\Omega_k} \leq \frac{ \mid \tau_3 \mid}{T_0-  C_{\tau_3}} (\sum_{k=1}^{n_{el}} B_{2k}) \|R_3^{n,\theta} \| \hspace{19mm}
 \end{multline} 
The terms containing truncation errors already have been estimated earlier during apriori error estimation. Now we estimate the remaining terms as follows:
\begin{equation}
\begin{split} 
& \sum_{k=1}^{n_{el}} \tau_2'(R_3^{h,n,\theta}, \bigtriangledown \cdot E^{A,n,\theta}_{\textbf{u}})_{\Omega_k}  \\
& =  \sum_{k=1}^{n_{el}} \tau_2' (R_3^{h,n,\theta}, \bigtriangledown \cdot e^{n,\theta}_{\textbf{u}})_{\Omega_k} -  \sum_{k=1}^{n_{el}} \tau_2' (R_3^{h,n,\theta}, \bigtriangledown \cdot E^{I,n,\theta}_{\textbf{u}})_{\Omega_k}  \\
& =  \sum_{k=1}^{n_{el}} \tau_2'  (\bigtriangledown \cdot e^{n,\theta}_{\textbf{u}}, \bigtriangledown \cdot e^{n,\theta}_{\textbf{u}})_{\Omega_k} -  \sum_{k=1}^{n_{el}} \tau_2' ( \bigtriangledown \cdot e^{n,\theta}_{\textbf{u}}, \bigtriangledown \cdot E^{I,n,\theta}_{\textbf{u}})_{\Omega_k}  \\
 \end{split}
\end{equation}
\begin{equation}
\begin{split} 
& \leq \sum_{k=1}^{n_{el}} \tau_2 \int_{\Omega_k} ( \bigtriangledown \cdot e^{n,\theta}_{\textbf{u}})^2 +  \sum_{k=1}^{n_{el}} \tau_2 \int_{\Omega_k} \mid  (\bigtriangledown \cdot e^{n,\theta}_{\textbf{u}} ) (\bigtriangledown \cdot E^{I,n,\theta}_{\textbf{u}})  \mid  \\ 
& \leq  C_{\tau_2} (\sum_{i=1}^d \|\frac{\partial e^{n,\theta}_{ui}}{\partial x_i}\|)^2 +  C_{\tau_2} (\sum_{i=1}^d \|\frac{\partial e^{n,\theta}_{ui}}{\partial x_i}\|) (\sum_{i=1}^d \|\frac{\partial E^{I,n,\theta}_{ui}}{\partial x_i}\| )  \\
 & \leq 2 C_{\tau_2} \sum_{i=1}^d \|\frac{\partial e^{n,\theta}_{ui}}{\partial x_i}\|^2 + \epsilon_1' C_{\tau_2} \sum_{i=1}^d \|\frac{\partial e^{n,\theta}_{ui}}{\partial x_i}\|^2  +  \frac{C_{\tau_2}}{\epsilon_1'} \sum_{i=1}^d \|\frac{\partial E^{I,n,\theta}_{ui}}{\partial x_i}\|^2   \\
& \leq C_{\tau_2}[ (2+ \epsilon_1')  \sum_{i=1}^d \| e^{n,\theta}_{ui}\|_1^2  + \frac{h^2}{\epsilon_1'} \sum_{i=1}^d (\frac{1+\theta}{2} \|u_i^{n+1}\|_2+ \frac{1-\theta}{2} \|u_i^{n}\|_2)^2 ]  \\
& \leq   (2+ \epsilon_1') C_{\tau_2} \| e^{n,\theta}_{\textbf{u}}\|_1^2 + h^2 \frac{C_{\tau_2}}{\epsilon_1'} \bar{C}_5 
 \end{split}
\end{equation}
where the parameter $\bar{C}_5$ comes for applying assumption $\textbf{(iv)}$. Now the terms involving trancation error can be estimated in slightly different way as we have done in the previous section. Let us present here a detailed derivation of one term only and the other follows the same way.
\begin{equation}
\begin{split}
(TE_2^{n,\theta}, E^{A,n,\theta}_{c}) & = (TE_2^{n,\theta}, e^{n,\theta}_c) - (TE_2^{n,\theta}, E^{I,n,\theta}_c) \\
& \leq \mid (TE_2^{n,\theta}, e^{n,\theta}_c) \mid + \mid (TE_2^{n,\theta}, E^{I,n,\theta}_c) \mid \\
& \leq \|TE_2^{n,\theta} \| (\|e^{n,\theta}_c\| + \|E^{I,n,\theta}_c \|)\\
& \leq \frac{1}{\epsilon_2'} \|TE_2^{n,\theta}\|^2 + \frac{\epsilon_2'}{2} (\|e^{n,\theta}_c\|^2 + \|E^{I,n,\theta}_c \|^2) \\
& \leq  \frac{1}{\epsilon_2'} \|TE_2^{n,\theta}\|^2 +  \frac{\epsilon_2'}{2} \{ \|e^{n,\theta}_c\|^2 + h^4 (\frac{1+\theta}{2} \|c^{n+1}\|_2 + \frac{1-\theta}{2} \|c^n\|_2)^2 \} \\
& \leq \frac{1}{\epsilon_2'} \|TE_2^{n,\theta}\|^2 +  \frac{\epsilon_2'}{2} \|e^{n,\theta}_c\|^2_1 +  h^4 \frac{\epsilon_2'}{2} \bar{C}_6
\end{split}
\end{equation}
Similarly the estimated result for the remaining term is
\begin{equation}
(\textbf{TE}_1^{n,\theta}, E^{A,n,\theta}_{\textbf{u}}) \leq  \frac{1}{\epsilon_2'} \|\textbf{TE}_1^{n,\theta}\|^2 +  \frac{\epsilon_2'}{2} \|e^{n,\theta}_{\textbf{u}}\|^2_1 +  h^4 \frac{\epsilon_2'}{2} \bar{C}_7
\end{equation}
 and this completes estimating all the terms of $RHS^A_1$. Now the $trilinear$ term in $RHS$ in (96) can be estimated as follows using property \textbf{(b)} of the $trilinear$ term as follows:
\begin{equation}
\begin{split}
c(e_\textbf{u}^{n},\textbf{u}^{n,\theta}, e_\textbf{u}^{n,\theta}) & \leq C \|e_\textbf{u}^{n}\| \|\textbf{u}^{n,\theta}\|_2 \|e_\textbf{u}^{n,\theta}\|_1\\
& \leq \bar{C}_8 \|e_\textbf{u}^{n}\| \|e_\textbf{u}^{n,\theta}\|_1\\
& \leq  \bar{C}_8 \|e_\textbf{u}^{n,\theta}\|^2
\end{split}
\end{equation}
The term $\|\textbf{u}^{n,\theta}\|_2$ is bounded by the virtue of assumption \textbf{(iv)}
and applying $Poincare$ inequality on last term in $RHS$ in (96) we have
\begin{equation}
\mu_l \|e_{\textbf{u}}^{n,\theta}\|^2 \leq \mu_l C_P \mid e_{\textbf{u}}^{n,\theta} \mid_1^2 \leq \mu_l C_P \|e_{\textbf{u}}^{n,\theta}\|^2
\end{equation}
Now this completes finding bounds for each term in the $RHS$ of (93). Putting common terms all together in the left hand side and multiplying them by $2$ and  then integrating both sides over $(t^n,t^{n+1})$ for $n=0,...,(N-1)$ , we will finally have 
\begin{multline}
\sum_{n=0}^{N-1} ( \|e^{n+1}_{\textbf{u}}\|^2-\|e^{n}_{\textbf{u}}\|^2)+\sum_{n=0}^{N-1} ( \|e^{n+1}_c\|^2-\|e^{n}_c\|^2) + \{ 2\mu_l-  2(2+\epsilon_1') C_{\tau_2}  \\
\quad - 2 C_8 - \mu_l C_P -\epsilon_2' \} \sum_{n=0}^{N-1} \int_{t^n}^{t^{n+1}} \|e^{n,\theta}_{\textbf{u}}\|_1^2 dt +(2 D_{\alpha}- \epsilon_2') \sum_{n=0}^{N-1} \int_{t^n}^{t^{n+1}} \|e^{n,\theta}_c\|^2_1 dt \\
\leq h^2 \sum_{n=0}^{N-1} \int_{t^n}^{t^{n+1}} \{ \bar{C}_1 \|\textbf{R}_1^{n,\theta}\| + \bar{C}_2 \|R_2^{n,\theta}\|+ \bar{C}_3 \|R_3^{n,\theta}\|+ \frac{2 C_{\tau_2}}{\epsilon_1'} \bar{C}_5 + h^2 \epsilon_2' (\bar{C}_6 + \bar{C}_7) \}dt \\
\quad  
+ 2 \frac{\mid \tau_1 \mid }{T_0- \rho C_{\tau_1}} [ (T+ \rho C_{\tau_1}) (\sum_{k=1}^{n_{el}} D_{B_{1k}}) + 2 \rho \sum_{i=1}^d \sum_{k=1}^{n_{el}} B_{1k}^i ]  \sum_{n=0}^{N-1} \int_{t^n}^{t^{n+1}} \|\textbf{R}_1^{n,\theta}\|^2 dt \\
+ 2 \frac{\mid \tau_3 \mid }{T_0- \rho C_{\tau_3}}[(T+ C_{\tau_3}) \sum_{k=1}^{n_{el}}D_{B_{3k}}+ 2 \sum_{k=1}^{n_{el}} B_{2k}] \sum_{n=0}^{N-1} \int_{t^n}^{t^{n+1}} \|R_3^{n,\theta}\|^2 dt +\\
\frac{2}{\epsilon_2'} \sum_{n=0}^{N-1} \int_{t^n}^{t^{n+1}}( \|\textbf{TE}^{n,\theta}_1\|^2+ \|TE^{n,\theta}_2\|^2) dt \hspace{20mm}
\end{multline}
Choose the arbitrary parameters including $C_{\tau_2}$ and the $Poincare$ constant $C_P$ in such a way that all the coefficients in the left hand side can be made positive.
Then taking minimum over the coefficients in the left hand side let us divide both sides by them. Using properties (15)-(16) associated with both implicit time discretisation scheme  and the fact that $\tau_1, \tau_3$ are of order $h^2$, we have arrived at the following relation:
\begin{equation}
\boxed{\|\textbf{u}-\textbf{u}_h\|_{\tilde{\textbf{V}}}^2  + \|c-c_h\|_{\tilde{\textbf{V}}}^2 \leq C'(\textbf{R}) (h^2+dt^{2r})}
\end{equation}
where
\begin{equation}
    r=
    \begin{cases}
      1, & \text{if}\ \theta=1 \hspace{1mm} for  \hspace{1mm} backward  \hspace{1mm} Euler  \hspace{1mm} rule \\
      2, & \text{if}\ \theta=0  \hspace{1mm} for  \hspace{1mm} Crank-Nicolson  \hspace{1mm} scheme
    \end{cases}
  \end{equation}
  This only completes one part of $aposteriori$ estimation and in the next part we combine the corresponding pressure part.\vspace{2mm}\\
\textbf{Second part}: Using the result (74) we can rewrite (73) in the following form:
\begin{equation}
b(\textbf{v}_h, I_hp-p_h)  = (\frac{\partial e_{\textbf{u}}}{\partial t}, \textbf{v}_{h})+  c(e_{\textbf{u}},\textbf{u},\textbf{v}_h)+
 c(\textbf{u}_h, e_{\textbf{u}},\textbf{v}_h)+ a_{NS}(e_{\textbf{u}},\textbf{v}_h)
\end{equation}
Integrating both sides with respect time
\begin{equation}
\begin{split}
\sum_{n=0}^{N-1} \int_{t^n}^{t^{n+1}} b(\textbf{v}_h, E_p^{A,n,\theta}) dt & = \sum_{n=0}^{N-1} \int_{t^n}^{t^{n+1}} \{ (\frac{e_{\textbf{u}}^{n+1}-e_{\textbf{u}}^n}{dt},\textbf{v}_{h})+ c(e_\textbf{u}^{n}, \textbf{u}^{n,\theta},\textbf{v}_h) +\\
& \quad  c(\textbf{u}_h^{n},e_{\textbf{u}}^{n,\theta}, \textbf{v}_h)+ a_{NS}(e_{\textbf{u}}^{n,\theta},\textbf{v}_h) \}dt\\
& = \sum_{n=0}^{N-1} \int_{t^n}^{t^{n+1}} \{ (\frac{e_{\textbf{u}}^{n+1}-e_{\textbf{u}}^n}{dt},\textbf{v}_{h})+ c(e_\textbf{u}^{n}, \textbf{u}^{n,\theta},\textbf{v}_h) +\\
& \quad  c(\textbf{u}^{n}, \textbf{u}^{n,\theta},\textbf{v}_h)- c(e_\textbf{u}^{n}, e_\textbf{u}^{n,\theta},\textbf{v}_h)+  a_{NS}(e_{\textbf{u}}^{n,\theta},\textbf{v}_h)\}dt
\end{split}
\end{equation} 
Now applying $Cauchy-Schwarz$'s inequality, $Young$'s inequality, property \textbf{(b)} of the $trilinear$ form $c(\cdot, \cdot, \cdot)$ and the above result (113) on (116) we have
\begin{equation}
\sum_{n=0}^{N-1} \int_{t^n}^{t^{n+1}} b(\textbf{v}_h, E_p^{A,n,\theta}) dt \leq \bar{C}'(\textbf{R})(h^2+ dt^{2r}) \|\textbf{v}_h\|_1
\end{equation}
Applying this result on (75) we have
\begin{equation}
\|I_hp-p_h\|^2_{L^2(L^2)} \leq \bar{C}''(\textbf{R})(h^2+ dt^{2r})
\end{equation} 
Now combining the results obtained in the first and second part and applying interpolation estimate on pressure interpolation term $E^I_p$, we finally arrive at the following
\begin{equation}
\boxed{\|\textbf{u}-\textbf{u}_h\|^2_{\tilde{\textbf{V}}} + \|p-p_h\|_{L^2(L^2)}^2  + \|c-c_h\|^2_{\tilde{\textbf{V}}} \leq \bar{C}(\textbf{R}) (h^2+ dt^{2r})}
\end{equation} 
Now this finally completes derivation of $aposteriori$ error estimation. 
\end{proof}
\begin{remark}
These estimations clearly imply that the scheme is $first$ order convergent in space with respect to total norm, whereas in time it is $first$ order convergent for backward Euler time discretization scheme and $second$ order convergent for Crank-Nicolson method.
\end{remark}

\section{Numerical Experiment}
In this section we verify the credibility of $ASGS$ method for this coupled transient $Navier$-$Stokes$-$VADR$ model through several numerical examples. Here we present a comparative study between standard Galerkin and $ASGS$ finite element method. We have considered two broad cases based on one way coupling and two-way or strong coupling. First case is further divided into three sub-cases consisting of different values Reynolds number and in the later one the viscosity of the fluid is taken to be dependent upon concentration of the solute and variable diffusion coefficients have been considered. This case too consists of two sub-cases involving different viscosity coefficients. \vspace{1mm}\\
Let us take $\Omega$ to be a square bounded domain (0,1) $\times$ (0,1).   Piecewise continuous linear finite element(P1) space is considered for approximating velocity, pressure and concentration. Now renaming the error in the following way we have examined the performances of both Galerkin and $ASGS$ methods.  
\begin{center}
$Total$ $error$ =  $\{\|\textbf{u}-\textbf{u}_h\|^2_{\tilde{\textbf{V}}} + \|p-p_h\|_{L^2(L^2)}^2  + \|c-c_h\|^2_{\tilde{\textbf{V}}}\}^{\frac{1}{2}}$
\end{center}
 The exact solutions for all the cases are taken as follows: \vspace{1mm}\\
$\textbf{u}=(e^{-t} x^2(x-1)^2y(y-1)(2y-1), -e^{-t}x(x-1)(2x-1)y^2(y-1)^2 )$, \vspace{1mm} \\
 $p=e^{-t}(3x^2+3y^2-2)$ and $c=e^{-t} x y (x-1)(y-1)$ \vspace{1mm} \\
 
 \begin{table}[]
\centering
\begin{tabular}{| *{6}{c|} }
    \hline
 Time &  Grid     & \multicolumn{2}{c|}{Galerkin method}
            & \multicolumn{2}{c|}{ASGS method}\\
      \hline       
  step & size     & Total error & RoC & Total error & RoC \\
 \hline
0.1& 10 $\times$ 10 & 0.158556 &  & 0.158435 & \\
   \hline
 0.05&    20 $\times$ 20 & 0.0833 & 0.928605  & 0.0833011 & 0.927481\\  
     \hline
  0.025&   40 $\times$ 40 & 0.0430609 & 0.95194  & 0.0430864 & 0.951103\\   
    \hline         
 0.0125&   80 $\times$ 80 & 0.0219347 & 0.973161  & 0.0219556 & 0.972645\\  
     \hline         
 0.00625&   160 $\times$ 160 & 0.0110526 & 0.98883  & 0.011068 & 0.988194\\    
    \hline      
\end{tabular}
\caption{Total error and Rate of convergence(RoC) under both Galerkin and $ASGS$ method for small Reynolds number(Re=50) at $T=1$}
    \end{table}
 
  \begin{table}[]
\centering
\begin{tabular}{| *{6}{c|} }
    \hline
 Time &  Grid     & \multicolumn{2}{c|}{Galerkin method}
            & \multicolumn{2}{c|}{ASGS method}\\
      \hline       
  step & size     & Total error & RoC & Total error & RoC \\
 \hline
0.1& 10 $\times$ 10 & 0.170253 &  & 0.158437 & \\
   \hline
 0.05&    20 $\times$ 20 & 0.0871451 & 0.966187  & 0.0833212 & 0.92715\\  
     \hline
  0.025&   40 $\times$ 40 & 0.043821 & 0.991797  & 0.0431014 & 0.950949\\   
    \hline         
 0.0125&   80 $\times$ 80 & 0.022057 & 0.990389  & 0.0219237 & 0.975243\\  
     \hline         
 0.00625&   160 $\times$ 160 & 0.011189 & 0.979155  & 0.011076 & 0.985054\\    
    \hline      
\end{tabular}
\caption{Total error and Rate of convergence(RoC) under both Galerkin and $ASGS$ method for small Reynolds number(Re=500) at $T=1$}
    \end{table}
 
 \begin{table}[]
\centering
\begin{tabular}{| *{6}{c|} }
    \hline
 Time &  Grid     & \multicolumn{2}{c|}{Galerkin method}
            & \multicolumn{2}{c|}{ASGS method}\\
      \hline       
  step & size     & Total error & RoC & Total error & RoC \\
 \hline
0.1& 10 $\times$ 10 & 0.226209 &  & 0.158438 & \\
   \hline
 0.05&    20 $\times$ 20 & 0.164603 & 0.458663  & 0.0833293 & 0.927026\\  
     \hline
  0.025&   40 $\times$ 40 & 0.0822173 & 1.00148  & 0.0431091 & 0.950832\\   
    \hline         
 0.0125&   80 $\times$ 80 & 0.0310324 & 1.37098  & 0.0219345 & 0.974791\\  
     \hline         
 0.00625&   160 $\times$ 160 & 0.022146 & 0.486729  & 0.0110245 & 0.992489\\    
    \hline      
\end{tabular}
\caption{Total error and Rate of convergence(RoC) under both Galerkin and $ASGS$ method for small Reynolds number(Re=5000) at $T=1$}
    \end{table}
 
  \begin{table}[]
\centering
\begin{tabular}{| *{6}{c|} }
    \hline
 Time &  Grid     & \multicolumn{2}{c|}{Galerkin method}
            & \multicolumn{2}{c|}{ASGS method}\\
      \hline       
  step & size     & Total error & RoC & Total error & RoC \\
 \hline
0.1& 10 $\times$ 10 & 0.159204 &  & 0.158826 & \\
   \hline
 0.05&    20 $\times$ 20 & 0.0834547 & 0.931812  & 0.0834583 & 0.928321\\  
     \hline
  0.025&   40 $\times$ 40 & 0.0430917 & 0.953584  & 0.043141 & 0.951997\\   
    \hline         
 0.0125&   80 $\times$ 80 & 0.021942 & 0.973715  & 0.0219748 & 0.973209\\  
     \hline         
 0.00625&   160 $\times$ 160 & 0.011122 & 0.980278  & 0.011022 & 0.993703\\    
    \hline      
\end{tabular}
\caption{Total error and Rate of convergence(RoC) under both Galerkin and $ASGS$ method for variable viscosity and diffusion coefficients (first sub-case) at $T=1$}
    \end{table}
 
  \begin{table}[]
\centering
\begin{tabular}{| *{6}{c|} }
    \hline
 Time &  Grid     & \multicolumn{2}{c|}{Galerkin method}
            & \multicolumn{2}{c|}{ASGS method}\\
      \hline       
  step & size     & Total error & RoC & Total error & RoC \\
 \hline
0.1& 10 $\times$ 10 & 0.236613 &  & 0.161085 & \\
   \hline
 0.05&    20 $\times$ 20 & 0.201906 & 0.22884  & 0.0855817 & 0.912444\\  
     \hline
  0.025&   40 $\times$ 40 & 0.128248 & 0.654755  & 0.0445594 & 0.941572\\   
    \hline         
 0.0125&   80 $\times$ 80 & 0.0495898 & 1.37082  & 0.0227193 & 0.971811\\  
     \hline         
 0.00625&   160 $\times$ 160 & 0.041146 & 0.269291  & 0.011148 & 1.027133\\    
    \hline      
\end{tabular}
\caption{Total error and Rate of convergence(RoC) under both Galerkin and $ASGS$ method for variable viscosity and diffusion coefficients (second sub-case) at $T=1$}
    \end{table}

$\textbf{(I)First case:}$ Here we have considered constant viscosity coefficient and therefore the coupled system becomes an one-way coupling. The importance behind considering this case is here that we want to verify the performance of $ASGS$ method for different Reynolds number. Here diffusion coefficients are also taken constant. \vspace{1mm}\\
$\textbf{(a)Small Reynolds number}$ The exact solutions remain same. The values of Reynolds number $Re$=50, diffusion coefficient $D$=2 and reaction coefficient $\beta$=0.01. \vspace{1mm}\\
Table 1 presents total errors and rates of convergence (RoC) of the coupled system for this case under Galerkin and $ASGS$ methods for different time steps $dt$ and grid sizes. It is clearly seen that both Galerkin and $ASGS$ method performs equally well for small Reynolds number. We can conclude the order of convergence for each of the methods is 1.\vspace{1mm}\\
$\textbf{(b) Medium Reynolds number:}$ For this case the values of coefficients are taken as $Re$=500, $D$=2 and $\beta$=0.01. Similar to the previous case table 2  represents the total errors and rates of convergence of the coupled system for this case under Galerkin and $ASGS$ methods  for different time steps $dt$ and grid sizes. In this case though both Galerkin and $ASGS$ method perform equally well and retain the desired first order convergence, but total error obtained in $ASGS$ method is less compared to that of Galerkin method. \vspace{1mm}\\
$\textbf{(c)Large Reynolds number}$ The values of the coefficients are considered as $Re$=5000, $D$=2 and $\beta$=0.01. Table 3 presents the total errors and rates of convergence of the coupled system for this case under Galerkin and $ASGS$ methods. It can be observed that Galerkin method behaves in somewhat oscillatory manner and it is not possible to conclude a definite order of convergence for this case, whereas $ASGS$ method performs consistently well at every time steps and grid sizes and rate of convergence in this case again turns out to be 1. \vspace{1mm}\\
$\textbf{(II)Second case:}$ Here we consider the viscosity to be dependent upon concentration and hence $Navier$-$Stokes$ and Transport equations are coupled in two-way manner. The proposed expression of concentration dependent viscosity is taken from \cite{RefP} and depending upon different viscosity coefficients we have divided this case into two sub-cases. In this case we have considered variable diffusion coefficients as follows: \vspace{1mm}\\
$D_1$= $e^{-t}y^2(y-1)^2(2y-1)^2x^4(x-1)^4$ and $D_2$= $e^{-t}x^2(x-1)^2(2x-1)^2y^4(y-1)^4$ \\
and the reaction coefficient $\beta$=0.01.\vspace{1mm}\\
$\textbf{(a) First sub-case:}$ The viscosity coefficient is $\mu(c)=0.00954 e^{27.93 \times 0.028 c}$. Table 4 presents the total errors and rates of convergence of the coupled system for this case under Galerkin and $ASGS$ methods for different time steps $dt$ and grid sizes. Both  Galerkin and $ASGS$ method performs equally well and order of convergence for both the methods is 1. \vspace{1mm}\\
$\textbf{(b) Second sub-case:}$ Here we have considered slightly small viscosity coefficient $\mu(c)=0.0000954 e^{27.93 \times 0.028 c}$. Table 5 presents the total errors and rates of convergence of the coupled system for this case under Galerkin and $ASGS$ methods. This table shows that Galerkin method performs poorly, whereas the $ASGS$ method performs far better and obtains the desired first order convergence.

\section{Conclusion}
This paper presents algebraic $subgrid$ $multiscale$ stabilized finite element analysis of transient $Navier$-$Stokes$ fluid flow equation strongly coupled with unsteady $VADR$ transport problem. Consideration of concentration dependent viscosity makes this time dependent coupling more accurate to model real life based contemporary problems. To ensure the efficiency of the stabilized finite element method for this model, both $apriori$ and $aposteriori$ error estimates have been derived in detail. It is essential to mention that the norm employed for error estimation consists of the full norms corresponding to each variable belonging to their respective spaces. Therefore it provides a wholesome information about convergence of the method. Theoretically the rate of convergence for both $apriori$ and $aposteriori$ error estimations is $O(h)$ in space and first and second order convergences have come out  for two implicit time discretization schemes viz. backward Euler and Crank-Nicolson methods respectively. The accuracy of the stabilized method has been numerically tested through considering two different kind of examples and various possible combinations among them. Numerical results both in tabular and figure representations show better performance of the stabilized $ASGS$ method than standard $Galerkin$ finite element method and verify theoretically established results too.\vspace{1cm}\\
{\large \textbf{Acknowledgement}} \vspace{2mm}\\
This work has been supported by grant from Innovation in Science Pursuit for Inspired Research (INSPIRE) programme sponsored and managed by the Department of Science and Technology(DST), Ministry of Science and Technology, Govt.of India.


\begin{thebibliography}{}
\bibitem{RefA}
A.N. Brooks, T.J.R. Hughes, Streamline upwind/Petrov–Galerkin formulations for convection dominated flows with particular emphasis on the incompressible Navier–Stokes equations, Computer Methods in Applied Mechanics and Engineering, 32: 199–259(1982).
\bibitem{RefB}
T.J.R. Hughes, L.P. Franca, G.M. Hulbert, A new finite element formulation for fluid dynamics: VIII. The Galerkin/least-squares method for advective–diffusive equations, Computer Methods in Applied Mechanics and
Engineering, 73: 173–189(1989).
\bibitem{RefC}
T.J.R. Hughes, Multiscale phenomena: Green’s functions, the Dirichlet-to-Neumann formulation, subgrid scale models, bubbles and the origins of stabilized methods, Computer Methods in Applied Mechanics and Engineering, 127: 387–401(1995).
\bibitem{RefD}
S.K. Hannani, M. Stanislas, P. Dupont, Incompressible Navier-Stokes computations with SUPG and GLS formulations — A comparison study, Computer Methods in Applied Mechanics and Engineering, 124: 153-170(1995).
\bibitem{RefE}
R. Codina, H.Coppola-Owen, P. Nithiarasu, C.B. Liu, Numerical comparison of CBS and SGS as stabilization techniques for the incompressible Navier–Stokes equations, International Journal for Numerical Methods in Engineering, 66:1672–1689 (2006).
\bibitem{RefF}
B.S. Krik, G.F. Carey, Development and validation of a SUPG finite element scheme for the compressible Navier–Stokes equations using a modified inviscid flux discretization, International Journal for Numerical Methods in Fluids, 57:265-293(2008).
\bibitem{RefG}
A. R$\acute{u}$sso, Streamline-upwind Petrov/Galerkin method (SUPG)
vs residual-free bubbles (RFB), Comput. Methods Appl. Mech. Engrg. 195: 1608–1620 (2006). 
\bibitem{RefH}
A. R$\acute{u}$sso, Bubble stabilization of the finite element methods for the linearized incompressible Navier–Stokes equations. Computer Methods in Applied Mechanics and Engineering, 132: 335–343(1996).
\bibitem{RefI}
R. Codina, O.C. Zienkiewicz, CBS versus GLS stabilization of the incompressible Navier-Stokes equations and the role of the time step as stabilization parameter, Communications in Numerical Methods in Engineering, 18: 99-112(2002).
\bibitem{RefJ}
D. Vassilev, I. Yotov, Coupling Stokes-Darcy flow with transport, SIAM J. Sci. Comput. 3661-3684(2009).
\bibitem{RefK}
A. Cesmelio$\breve{g}$lu, P. Chidyagwai, B. Rivi$\grave{e}$re, Continuous and discontinuous finite element methods for coupled surface-subsurface flow and transport problems.
\bibitem{RefL}
A. Cesmelio$\breve{g}$lu, B. Rivi$\grave{e}$re, Existence of a weak solution for the fully coupled Navier-Stokes/Darcy-transport problem, J. Differential Equations, 252, 4138-4175(2012).
\bibitem{RefM}
H. Rui, J. Zhang, A stabilized mixed finite element method for coupled Stokes and Darcy flows with transport, Comput. Methods Appl. Mech. Engrg. 315, 169-189(2017).
\bibitem{RefN}
M. Chowdhury, B.V. Rathish Kumar, Apriori and aposteriori error estimation of Subgrid multiscale stabilized finite element method for coupled unified Stokes-Brinkman/Transport model, arxiv (pre-print), math.AP, 2004.01782.
\bibitem{RefO}
R. Codina, Comparison of some finite element methods for solving the
diffusion-convection-reaction equation, Comput. Methods Appl. Mech. Engrg. 156: 185-210 (1998).
\bibitem{RefP}
J. Chirife, M. P. Buera, A Simple Model for Predicting the Viscosity of Sugar and Oligosaccharide Solutions, Journal of Food Engineering 33, 221-236(1997).
\bibitem{RefQ}
B. Rivi$\grave{e}$re, M. F. Wheeler, A Discontinuous Galerkin Method Applied to Nonlinear Parabolic Equations, Discontinuous Galerkin Methods, Springer, pp. 231-244(2000).
\bibitem{RefR}
J. Blasco, R. Codina, Space and time error estimates for a first order pressure stabilized finite element method for the incompressible Navier-Stokes equations, Applied Numerical Mathematics 28, 475-497(2001).
\bibitem{RefS}
M. Chowdhury, B.V.R. Kumar, A priori and a posteriori error estimation for finite element approximation of advection-diffusion-reaction equation with spatially variable coefficients, arxiv (pre-print), math.AP, 1811.05283.
\bibitem{RefT}
M. Chowdhury, B.V.R. Kumar, On subgrid multiscale stabilized finite element method for
advection-diffusion-reaction equation with variable coefficients, Applied Numerical Mathematics 150, 576-586(2020). 
\bibitem{RefU}
G.R. Barrenechea, E. Castillo, R. Codina, Time-dependent semi-discrete analysis of the viscoelastic fluid flow problem using a variational multiscale stabilised formulation, IMA Journal of Numerical Analysis, 1-25(2018).
\bibitem{RefV}
S. Badia, R. Codina, Unified stabilized finite element formulations for the Stokes and Darcy problems, SIAM J. Numer. Anal., 47(3), 1971–2000(2009).
\bibitem{RefW}
R. Du, Z. Liu, A lattice Boltzmann model for the fractional advection–diffusion equation coupled with incompressible Navier–Stokes equation, Applied Mathematics Letters, 101(2020).
\bibitem{RefX}
X. Yua, K. Regenauer-Lieba, F. Tian, A hybrid immersed boundary-lattice Boltzmann/finite difference method for coupled dynamics of fluid flow, advection, diffusion and adsorption in
fractured and porous media, Computers and Geosciences, 128, 70-78(2019).
\bibitem{RefA1}
G. Hauke, A simple subgrid scale stabilized method for the advection–diffusion-reaction equation, Comput. Methods Appl. Mech. Engrg., 191, 2925–2947(2002).
\bibitem{RefA2}
R. Codina, A stabilized finite element method for generalized stationery incompressible flows, Comput. Methods Appl. Mech. Engrg., 190, 2681-2706(2001).
\bibitem{RefA3}
R. Codina, J. Blasco, Stabilized finite element method for the transient Navier-Stokes equations based on a pressure gradient projection, Comput. Methods Appl. Mech. Engrg., 182, 277-300(2000).
\bibitem{RefB1}
A. Rasam, S. Wallin, G. Brethouwer , A. V. Johansson, Large eddy simulation of channel flow with and without periodic constrictions using the explicit algebraic subgrid-scale model, Journal of Turbulence, 15(11), 752–775(2014).
\bibitem{RefB2}
R. Codina, On stabilized finite element methods for linear systems of convection-diffusion-reaction equations, Comput. Methods Appl. Mech. Engrg., 188, 61-82 (2000).
\bibitem{RefB3}
M. Bischoff, K. Bletzinger, Improving stability and accuracy of Reissner–Mindlin
plate finite elements via algebraic subgrid scale stabilization, Comput. Methods Appl. Mech. Engrg., 193, 1517–1528 (2004).
\bibitem{RefB4}
O. Guasch, R. Codina, An algebraic subgrid scale finite element method for the
convected Helmholtz equation in two dimensions with applications in aeroacoustics, Comput. Methods Appl. Mech. Engrg., 196, 4672–4689(2007).
\bibitem{RefC1}
E. Castillo, R. Codina, Dynamic term-by-term stabilized finite element formulation using orthogonal subgrid-scales for the incompressible Navier-Stokes problem, Comput. Methods Appl. Mech. Engrg., 349, 701-721(2019).
\bibitem{RefC2}
R. Codina, J. Blasco, Analysis of a stabilized finite element approximation of the transient
convection-diffusion-reaction equation using orthogonal subscales, Comput Visual Sci, 4, 167–174 (2002).
\bibitem{RefC3}
J. Baigesa, R. Codina, Variational Multiscale error estimators for solid mechanics adaptive simulations: An Orthogonal Subgrid Scale approach, Comput. Methods Appl. Mech. Engrg., 325, 37-55(2017).
\bibitem{RefC4}
R. Codina, Analysis of a stabilized finite element approximation of the Oseen equations using orthogonal subscales, Applied Numerical Mathematics, 58, 264-283 (2008).
\bibitem{RefC5}
C. Bayona, J. Baiges, R. Codina, Variational multiscale approximation of the one-dimensional forced Burgers equation: The role of orthogonal subgrid scales in turbulence modeling, Int J Numer Meth Fluids.,86, 313–328(2018).
\bibitem{RefD1}
V. Girault, P.A. Raviart, Finite Element Methods for Navier-Stokes Equations: Theory and Algorithms, vol. 5 of Springer series in computational mathematics. Springer, Berlin (1986).
\bibitem{RefD2}
E. Burman, M. A. Fernández, Continuous interior penalty finite element method for the time-dependent Navier–Stokes equations: space discretization and convergence, Numer. Math., 107,39–77(2007).
\bibitem{RefD3}
J. A. Wheeler, M. F. Wheeler, I. Yotovc, Enhanced velocity mixed finite element methods for
flow in multiblock domains, Computational Geosciences 6: 315–332(2002).


 














\end{thebibliography}
\end{document}